\documentclass[12pt,reqno]{amsart}
\usepackage{amsmath,amsthm,amssymb,amsfonts,amscd}
\usepackage{mathrsfs}
\usepackage{bbm}
\usepackage{bbding}
\usepackage[backref=page]{hyperref}
\usepackage{geometry}
\usepackage{color}
\usepackage{xcolor}
\usepackage{picture,epic}
\usepackage{tikz}

\numberwithin{equation}{section}

\theoremstyle{plain}
\newtheorem{theorem}{Theorem}[section]
\newtheorem{lemma}[theorem]{Lemma}
\newtheorem{corollary}[theorem]{Corollary}
\newtheorem{proposition}[theorem]{Proposition}

\theoremstyle{definition}
\newtheorem{conjecture}[theorem]{Conjecture}

\theoremstyle{remark}
\newtheorem{remark}[theorem]{Remark}

\renewcommand{\Re}{\operatorname{Re}}
\renewcommand{\Im}{\operatorname{Im}}
\newcommand{\vol}{\operatorname{vol}}

\newcommand{\sym}{\operatorname{sym}}
\newcommand{\Sym}{\operatorname{Sym}}

\newcommand{\SL}{\operatorname{SL}}

\renewcommand{\mod}{\operatorname{mod}\ }

\newcommand{\dd}{\mathrm{d}}

\makeatletter

\newcommand{\Rmnum}[1]{\expandafter\@slowromancap\romannumeral #1@}
\makeatother

\title[Mixed fourth moments of automorphic forms]{Mixed fourth moments of automorphic forms and the shifted moments of $L$-functions}
\author{Chengliang Guo}
\date{\today}

\address{Mathematical Research Center \\ Shandong University \\ Jinan \\Shandong 250100 \\China}
\email{chengliang.guo@mail.sdu.edu.cn}

\begin{document}

\begin{abstract}
 		In this article, we study the mixed fourth moments of Hecke--Maass cusp forms and Eisenstein series with type $(2, 2)$. Under the assumptions of the Generalized Riemann Hypothesis (GRH) and the Generalized Ramanujan Conjecture (GRC), we establish asymptotic formulas for these moments. Our results give an interesting non-equidistribution phenomenon over the full fundamental domain. In fact, this independent equidistribution should be true in a compact set. We further investigate this behaviour by examining a truncated version involving truncated Eisenstein series. Additionally, we propose a conjecture on the joint value distribution of Eisenstein series. The proofs are based on the bounds of the shifted mixed moments of $L$-functions. 
\end{abstract}

\thanks{This work was supported by  the National Key R\&D Program of China (No. 2021YFA1000700).}
\nocite{*}
\maketitle

\tableofcontents

\section{Introduction}
The study of the value distribution of automorphic forms is a central problem in analytic
number theory and arithmetic quantum chaos. Let $\mathbb{H}:=\{x+iy : x \in \mathbb{R}, y>0\}$ be the upper half space. Let $\Gamma = \SL(2,\mathbb{Z})$ and  $\mathbb{X} = \Gamma\backslash\mathbb{H}$. On the quotient space $\mathbb{X}$, we have the Petersson inner product defined by $\langle f , g\rangle = \int_{\mathbb{X}}f(z)\overline{g(z)}\dd \mu z$ where $\dd\mu z = \frac{\dd x \dd y}{y^2}$.

The Laplacian is given by $\Delta=-y^2 (\frac{\partial^2}{\partial x^2}+\frac{\partial^2}{\partial y^2})$, which has both discrete and continuous spectra.
The discrete spectrum consists of the constants and the space of Maass cusp forms. In particular, we may choose an orthonormal basis $\{\phi_k\}$ of Hecke--Maass cusp forms. Throughout this article, we assume that each Hecke--Maass cusp form $f$ is normalized so that $\langle f, f \rangle = 1$. Eisenstein series play a fundamental role in describing the continuous spectrum. But they are not $L^2$-integrable. In this article, we write $E_T(z) = E(z, \frac{1}{2}+iT)$ for standard Eisenstein series in the critical line.

For $L^2$-norm of Hecke--Maass cusp forms, this is the well-known quantum unique ergodicity (QUE) conjecture of Rudnick and Sarnak \cite{MR1266075}. QUE was solved by breakthrough papers of Lindenstrauss \cite{MR2195133} and
 Soundararajan \cite{MR2680500}. For $\psi(z) \in C_{c}^{\infty}(\mathbb{X})$, they proved
\begin{equation}\label{QUE-Maass}
    \int_{\mathbb{X}}\psi(z)\phi^2(z) \frac{\dd x \dd y}{y^2} \rightarrow  \int_{\mathbb{X}}\psi(z)\frac{3}{\pi} \frac{\dd x \dd y}{y^2}
\end{equation}
 as the spectral parameter $t_{\phi}$ goes to infinite. If we replace the Hecke--Maass cusp forms to normalized Eisenstein series Luo and Sarnak \cite{MR1361757} gave a similar asymptotic formula which is
\begin{equation}\label{QUE-Eisenstein}
    \int_{\mathbb{X}}\psi(z)|E_{T}(z)|^2 \frac{\dd x \dd y}{y^2} \rightarrow  \int_{\mathbb{X}}\psi(z)\frac{3}{\pi} \frac{\dd x \dd y}{y^2} \times \log (\frac{1}{4}+T^2)
\end{equation}
as $T$ goes to infinite. Here $\sqrt{\log (\frac{1}{4}+T^2)}$ is the mass of $E_{T}(z)$.

Studying the higher moments value distribution of the automorphic forms is an interesting topic.  For $L^4$-norm with non-compact domain, Blomer--Khan--Young \cite{MR3127809} formulated the conjecture
\[
\int_{\mathbb{X}}\phi^4(z)\frac{\dd x\dd y}{y^2} \rightarrow \frac{9}{\pi}
\]
as $t_{\phi}$ goes to infinite. Under GLH, Buttcane and Khan \cite{MR3647437} proved above conjecture. Recently Ki \cite{ki2023l4normssignchangesmaass} proved the sharp upper bound called $\|\phi\|_4 \ll t_{\phi}^\varepsilon$. For the dihedral Maass forms case, Humphries and Khan \cite{MR4081056} proved the conjecture unconditionally and they also established an upper bound $\|\phi\|_4 \ll t_{\phi}^{3/152+\varepsilon}$ \cite{MR4923689} for more general congruence subgroup $\Gamma$. If we replace the Hecke--Maass cusp forms by Eisenstein series, two different unconditional asymptotic formulas were established by Djanković--Khan \cite{MR4184616}, \cite{djanković2024fourthmomenttruncatedeisenstein} due to regularized type  and truncated type since Eisenstein series are not square-integrable. They deduced
\[
\int_{\mathbb{X}}^{reg}|E(z,\frac{1}{2}+iT)|^4\frac{\dd x\dd y}{y^2} \rightarrow \frac{72}{\pi} (\log T)^2, 
\]
and

\[
\int_{\mathbb{X}}^{reg}|E^A(z,\frac{1}{2}+iT)|^4\frac{\dd x\dd y}{y^2} \rightarrow \frac{36}{\pi}(\log T)^2.
\]
Here the truncated Eisenstein series are defined by (\ref{truncated-Eisenstein-series}).

From the prediction that the values of
distinct Hecke–Maass cusp forms should behave like independent random waves, Hua--Huang--Li established the conjecture of all higher moments of Hecke--Maass forms \cite[Conjecture 1.3]{hua2024jointvaluedistributionheckemaass}. They examined the conjecture about the mixed moments of type $(2, 2)$ and proved it under GRH and GRC. That is
\[
\int_{\mathbb{X}}f^2(z)g^2(z) \frac{\dd x \dd y}{y^2} \rightarrow \frac{3}{\pi}
\]
as $\min\{t_f,t_g\}$ goes to infinite.

\subsection{Mixed fourth moments problem of automorphic forms}
It is natural to extend the value distribution conjecture to Eisenstein series. As a special case, we consider the following mixed moments of type $(2,2)$
\begin{equation}
\int_{\mathbb{X}}\phi^2(z)|E_{T}(z)|^2 \frac{\dd x\dd y}{y^2}.
\end{equation}
Since Eisenstein series is not square-integrable,  this subtlety leads to phenomena quite different from those observed in the joint value distribution of Hecke--Maass cusp forms over non-compact domains. Our main result is 
\begin{theorem}\label{Nonequidistribution-theorem}Assume GRH and GRC. Let $\phi$ be a Hecke--Maass form with spectral parameter $t_{\phi}$. For $T \geq 1$ we get
\begin{equation}
\int_{\mathbb{X}}\phi^2(z)|E_{T}(z)|^2 \frac{\dd x \dd y}{y^2} = \frac{3}{\pi}[\log (\frac{1}{4} + t_{\phi}^2) + \log (\frac{1}{4} + T^2)] + \mathcal{O}\left((\log (T + t_{\phi}))^{1/2+\varepsilon}\right)
\end{equation}
as $\min\{t_{\phi}, T\} \rightarrow \infty$.

\end{theorem}
We find that neither Hecke--Maass forms nor Eisenstein series exhibit equidistribution when their spectral parameters tend to infinity at a small power of one another.
\begin{remark}
Note that this result does not conflict with quantum unique ergodicity (QUE). From the perspective of QUE for Hecke--Maass forms, the function $|E_{T}(z)|^2$ cannot serve as a suitable test function because it is not integrable. In the context of the QUE for Eisenstein series, we have
    \begin{equation}
\int_{\mathbb{X}}\phi^2(z)|E_{T}(z)|^2 \frac{\dd x \dd y}{y^2} \sim \frac{3}{\pi} \log (\frac{1}{4} + T^2)
\end{equation}
since we take $\phi$ fixed as a test function.
\end{remark}

 We always assume that $T \leq t_{\phi}$. The crucial case is $t_{\phi} = T$. In our previous work \cite[Proposition 7.1, equation (7.1)]{MR4904450}, by regularized Plancherel formula, we have
\[
\langle \phi^2 , |E_{T}|^2\rangle = \mathcal{R}(\phi,E_T)
+ \sum\limits_{j\geq 1}\langle \phi_{j} , \phi^{2}\rangle\langle |E_{T}^{2}|, \phi_{j}\rangle + \frac{1}{4\pi}\int_{\mathbb{R}}\langle E_{\tau}, \phi^{2}\rangle\langle |E_{T}|^{2},E_{\tau}\rangle_{reg}\dd \tau 
\]
where 
\begin{equation}\label{Main-term}
\begin{aligned}
         \mathcal{R}(\phi, E_T) =   &\frac{1}{\xi(2)}\overline{[ \frac{\Lambda^{\prime}(1,\Sym^{2}\phi)}{\Lambda(1,\Sym^{2}\phi)}  +2\Re \frac{\xi^{\prime}(1+2iT)}{\xi(1+2iT)} - \frac{2\xi^{\prime}(2)}{\xi(2)} + a_{0}]}\\
         & + \frac{\xi(2iT)}{\xi(1+2iT)}\langle \phi^2 , E(z,1+2iT)\rangle + \overline{\frac{\xi(2iT)}{\xi(1+2iT)}}\langle \phi^2 , E(z,1-2iT)\rangle.
\end{aligned}
\end{equation}
\begin{remark}In Appendix \ref{sec:7}, we provide an alternative computation showing that the complex terms involving $\phi$ arise from some certain integrals related to Hecke--Maass forms, while the remainder term containing the logarithmic derivative of the zeta function originates from the Maass--Selberg relation.
\end{remark}
By explicit calculation and Stirling's formula, we get
\begin{multline}
   \mathcal{R}(\phi, E_T) = \frac{3}{\pi}[\log (\frac{1}{4} + t_{\phi}^2) + \log (\frac{1}{4} + T^2) + 2\frac{L^{\prime}}{L}(1,\Sym^2 \phi)+4\Re\frac{\zeta'}{\zeta}(1+2iT)] +\mathcal{O}(1) \\
+ \frac{\xi(2iT)}{\xi(1+2iT)}\langle \phi^2 , E(z,1+2iT)\rangle + \overline{\frac{\xi(2iT)}{\xi(1+2iT)}}\langle \phi^2 , E(z,1-2iT)\rangle. 
\end{multline}


    
\begin{remark}\label{Rankin-Selberg-trick} By unfolding trick and under the assumptions of GRH and GRC, we have
\begin{equation*}
    \begin{aligned}
         \langle \phi^2 , E(z,1+2iT)\rangle
         & \ll\frac{L(1-2iT,\Sym^2\phi)\zeta(1-2iT)}{2L(1,\Sym^2 \phi)\zeta(2-4iT)}\frac{\exp(-\frac{\pi}{2}(|T+t_\phi| + |T-t_{\phi}| - 2t_{\phi}))}{(1+|T|)^{1/2}}\\
         & \ll \frac{(\log(T + t_{\phi}))^{\varepsilon}}{(1+|T|)^{1/2-\varepsilon}}.
    \end{aligned}
\end{equation*}
When $T$ is small, it cannot be directly absorbed into the error term, as a sharp bound of $L(1,\Sym^2 \phi)$ is required.

\end{remark}
\begin{remark}For Theorem \ref{Nonequidistribution-theorem}, we have the unconditional result if $t_{\phi} \ll T^{1-\varepsilon}$ (near the case of QUE for Eisenstein series). In this case, for the discrete spectrum, we get
\begin{equation}
    \begin{aligned}
        \sum\limits_{j\geq 1}&\langle \phi_{j} , \phi^{2}\rangle\langle |E_{T}^{2}|, \phi_{j}\rangle \\
        &\ll \sum_{t_j \ll 2t_{\phi}+T^{\varepsilon}}|\langle \phi_{j} , \phi^{2}\rangle\langle |E_{T}^{2}|, \phi_{j}\rangle|\\
        & \leq \left(\sum_{t_j \ll 2t_{\phi}+T^{\varepsilon}}|\langle |E_{T}^{2}|, \phi_{j}\rangle|^2\right)^{1/2}\left(\sum_{t_j \ll 2t_{\phi}+T^{\varepsilon}}|\langle \phi^2, \phi_{j}\rangle|^2\right)^{1/2}\\
        & \leq \left(\sum_{t_j \ll 2t_{\phi}+T^{\varepsilon}}|\langle |E_{T}^{2}|, \phi_{j}\rangle|^2\right)^{1/2} \|\phi\|_4^2\\
        & \ll T^{-\varepsilon}.
    \end{aligned}
\end{equation}
For the final inequality we use $\sum_{t_j \ll T^{1-\varepsilon}}|\langle |E_{T}^{2}|, \phi_{j}\rangle|^2 \ll T^{-\varepsilon}$ (see \cite[Section 3.6]{MR4081056}) and $\|\phi\|_4 \ll t_{\phi}^{\varepsilon}$. Then we get in this case 
    \[
    \langle \phi^2 , |E_{T}|^2\rangle = \mathcal{R}(\phi,E_T) +\mathcal{O}\left(T^{-\varepsilon}\right).
    \]
\end{remark}

When we restrict the moments in a compact set,  the above independent equidistribution theorem is expected to hold. If we replace Eisenstein series by the truncated Eisenstein series defined in (\ref{truncated-Eisenstein-series}), we indeed confirm such independent equidistribution.
\begin{theorem}\label{equidistribution-phenomenon}Assume GRH and GRC.  Let $A > 1$ be fixed 
and $T \geq 1$. Let $\phi$ be a Hecke--Maass form with spectral parameter $t_{\phi}$ and $E_{T}^A(z)$ be the truncated Eisenstein series. Then we get
\begin{equation}
\int_{\mathbb{X}}\phi^2(z)|E_{T}^{A}(z)|^2 \frac{\dd x \dd y}{y^2} = \frac{3}{\pi}\log (\frac{1}{4} + T^2) + \mathcal{O}_{A,\varepsilon}( (\log T)^{1/2+\varepsilon})
\end{equation}
as $\min\{t_{\phi}, T\} \rightarrow \infty$.
\end{theorem}

\begin{remark}For some suitable ranges of $T$ and $t_{\phi}$, it may be possible to weaken our assumptions or improve the error term. However, since our aim is to establish an asymptotic formula valid in all ranges, we assume the full strength of GRH and GRC. In the rest of the paper,  especially in the proof of Theorem (\ref{Nonequidistribution-theorem}) and Theorem (\ref{equidistribution-phenomenon}), we  focus mainly on the crucial case $T \sim t_{\phi}$. The treatment of other ranges does not present additional difficulties.
\end{remark}
\subsection{Conjecture on joint value distribution of automorphic forms}In this subsection, we conclude the conjecture of value distibution of automorphic forms.

The Hecke--Maass form $f_i$ are normalized by 
\[
\frac{1}{\vol(\mathbb{X})}\int_{\mathbb{X}}|f_i(z)|^2 \frac{\dd x\dd y}{y^2} = 1
\]
and the real Eisenstein series  
\[
\widetilde{E_{T}}(z) := \sqrt{\frac{\vol(\mathbb{X})}{\log(\frac{1}{4}+T^2)}}\frac{\xi(1+2iT)}{|\xi(1+2iT)|}E_{T}(z). 
\]
We also recall the moments of a standard real Gaussian random variable 
\begin{equation}
    \begin{aligned}
       C_n = \frac{1}{\sqrt{2\pi}}\int_{-\infty}^{+\infty}x^ne^{-\frac{x^2}{2}}\dd x = \left\{\begin{array}{lr}
          (n-1)!!   ,&   2 \mid n,\\
            0 ,&   2 \nmid n.
        \end{array}
        \right.
    \end{aligned}
\end{equation}

Berry \cite{MR489542} suggested that eigenfunctions
for chaotic systems are modeled by random waves, it is believed that eigenfunctions on a
compact hyperbolic surface have a Gaussian value distribution as the eigenvalue tends to
infinity, and the moments of an $L^2$-normalized eigenfunction should be given by the Gaussian moments. More precisely, we have the following Gaussian moments conjecture (see e.g. Humphries \cite[Conjecture 1.1]{MR3831279}).
\begin{conjecture}\label{Conjecture-random-wave}
    For any $\psi(z) \in C_{c}^{\infty}(\mathbb{X})$, we have
\[
\int_{\mathbb{X}}\psi(z) f_i^{n}(z) \frac{\dd x \dd y}{y^2} \sim C_n\int_{\mathbb{X}}\psi(z)\frac{\dd x\dd y}{y^2}
\]
 as $t_{f_i}$ goes to infinite.  
\end{conjecture}

    We choose an orthogonal family of Hecke--Maass cusp forms $\{f_i\}_{i=1}^{I}$ with $\langle f_i ,f_{i'}\rangle = 0$ if $i \neq i'$. Recently, Hua--Huang--Li formulated the following conjecture \cite[Conjecture 1.3]{hua2024jointvaluedistributionheckemaass}.

\begin{conjecture}\label{Conjecture-Hua-Huang-Li}
  For positive integer $a_i$, Then $\{f_i^{a_i}\}_{i = 1}^{I}$ are statistically independent; that is for any $\psi(z) \in C_{c}^{\infty}(\mathbb{X})$, we have
\[
\int_{\mathbb{X}}\psi(z) \prod_{i = 1}^{I}f_i^{a_i}(z) \frac{\dd x \dd y}{y^2} \sim \prod_{i = 1}^{I}C_{a_i}\int_{\mathbb{X}}\psi(z)\frac{\dd x\dd y}{y^2}
\]
 as $\min\{t_{f_1}, \cdots,t_{f_I}\}$ goes to infinite. 
\end{conjecture}

Theorem \ref{equidistribution-phenomenon} suggests that, when normalized by their mass, Hecke--Maass cusp forms and Eisenstein series behave like independent random waves in a compact set. Let  $J \geq 1$. We choose an orthogonal family  we choose an orthogonal family of Eisenstein series $\{\widetilde{E_{T_{j}}}\}_{j = 1}^{J}$ with $|T_{j} - T_{j'}| \geq 1$ if $j \neq j'$.

We now state the following conjecture :
\begin{conjecture}\label{Conjecture-Joint-value-distribution}For positive integer $a_i, b_j$, Then $\{\{f_i^{a_i}\}_{i = 1}^{I} , \{\widetilde{E_{T_j}}\}_{j =1}^{J}\}$ are statistically independent; that is for any $\psi(z) \in C_{c}^{\infty}(\mathbb{X})$, we have
\[
\int_{\mathbb{X}}\psi(z) \prod_{i = 1}^{I}f_i^{a_i}(z)\prod_{j =1}^{J}\widetilde{E_{T_j}}^{bj}(z) \frac{\dd x \dd y}{y^2} \sim \prod_{i = 1}^{I}C_{a_i}\prod_{j = 1}^{J}C_{b_j}\int_{\mathbb{X}}\psi(z)\frac{\dd x\dd y}{y^2}
\]
 as $\min\{t_{f_1}, \cdots,t_{f_I}, T_1,\cdots T_{J}\}$ goes to infinite.   
\end{conjecture}
\begin{remark}The condition $|T_{j} - T_{j'}| \geq 1$ is necessary to ensure the distinguishability of the corresponding Eisenstein series. We will discuss more details in Appendix \ref{sec:6} in the view of the mixed moments family. As the simplest case, we find the decorrelation (for the case $I = 0$, $J = 2$, $b_1 = b_2 = 1$) fails without this restriction. See Proposition \ref{decorrelation-Eisenstein-series}.
\end{remark}

\subsection{Applications to the moments of Rankin--Selberg $L$-functions}
By Theorem \ref{Nonequidistribution-theorem}, we obtain an estimate for the moments of Rankin--Selberg $L$-functions. We consider an alternative decomposition of the mixed fourth moment
\[
\langle \phi^2 , |E_T|^2\rangle = \langle \phi E_{T} , \phi E_T\rangle.
\]
Using spectral decomposition, we have
\[
\mathcal{R}(\phi,E_{T}) + \sum_{j}\langle \phi^2, \phi_j\rangle\langle |E_{T}|^2, \phi_j\rangle + \mathcal{E}_1 = \sum_{j}|\langle \phi E_{T},\phi_j\rangle|^2 + \mathcal{E}_2
\]
where $\mathcal{E}_1, \mathcal{E}_2$ are the contribution of continuous spectrum. We can bound $\mathcal{E}_1, \mathcal{E}_2$ by a power saving error term unconditionally except $T \ll t_{\phi}^{\varepsilon}$. If $T \ll t_{\phi}^{\varepsilon}$, we need a subconvexity bound of $L(\frac{1}{2},\Sym^2 \phi)$ which also appears in the QUE problem (see (\ref{continous-spectrum})). 

The discrete spectrum in the right side contributes
\[
\frac{\pi^2}{4L(1,\Sym^2 \phi)|\zeta(1+2iT)|^2}\sum_{j}\frac{|L(\frac{1}{2}+iT, \phi\times \phi_j)|^2 }{L(1,\Sym^2 \phi_j)}\frac{ |\prod_{\pm_1}\prod_{\pm_2}\Gamma(\frac{1/2+iT \pm_1 it_\phi \pm_2 it_j}{2})|^2}{|\Gamma(\frac{1}{2}+iT)|^2|\Gamma(\frac{1}{2}+it_{\phi})|^2|\Gamma(\frac{1}{2}+it_{j})|^2}.
\]
We define the innermost Gamma factor by $\mathcal{\gamma}(t_j,t_{\phi},T)$ then we have the following result.
\begin{corollary}Assume GRH and GRC. We have
    \begin{multline}
        \frac{\pi^2}{4L(1,\Sym^2 \phi)|\zeta(1+2iT)|^2}\sum_{j}\frac{|L(\frac{1}{2}+iT, \phi\times \phi_j)|^2 }{L(1,\Sym^2 \phi_j)}\mathcal{\gamma}(t_j,t_{\phi},T)\\
        =\frac{3}{\pi}[\log (\frac{1}{4} + t_{\phi}^2) + \log (\frac{1}{4} + T^2)] + \mathcal{O}((\log (T + t_{\phi}))^{1/2+\varepsilon})
    \end{multline}
    as $\min\{t_{\phi }, T\}$ goes to infinity.
\end{corollary}

\begin{remark}Naturally, we use this decomposition and reduce the $(2,2)$--mixed moments to Rankin--Selberg $L$-function. By applying spectral large sieve inequality, we readily obtain 
    \[
    \frac{\pi^2}{4L(1,\Sym^2 \phi)|\zeta(1+2iT)|^2}\sum_{j}\frac{|L(\frac{1}{2}+iT, \phi\times \phi_j)|^2 }{L(1,\Sym^2 \phi_j)}\mathcal{\gamma}(t_j,t_{\phi},T) \ll (t_{\phi}T)^{\varepsilon}.
    \]
This implies unconditionally that
\[
\int_{\mathbb{X}}\phi^2(z)|E_{T}(z)|^2 \frac{\dd x \dd y}{y^2} \ll (t_{\phi}T)^{\varepsilon}.
\]
However, obtaining a sharp logarithmic asymptotic formula unconditionally appears to present substantial difficulties.
\end{remark}



\subsection{Fractional mixed moments problem of $L$-functions}Moments of $L$-functions represent a central topic in analytic number theory, with the estimation of mixed moments for certain families of $L$-functions being particularly important due to their applications in arithmetic quantum chaos. For instance, such estimates play a key role in quantum variance problem for dihedral Maass forms \cite{MR4510120} and quantum unique ergodicity problem of automorphic forms with half integer weight \cite{MR4057145}.

To get Theorem \ref{Nonequidistribution-theorem} and Theorem \ref{equidistribution-phenomenon}, we establish the following Theorem.
\begin{theorem}\label{Mixed-moments-Soundararajan's-method-shifted}
Assume GRH and GRC. Let $\phi$ be a Hecke--Maass form with spectral parameter $t_{\phi}$ and suppose $1\leq t_{\phi}^{\delta}\leq T \leq t_{\phi}$ for some $\delta > 0$. Let $\varepsilon > 0$
 be sufficiently small, $T^{1-\varepsilon}
  \leq X\leq 3t_{\phi}$, and $X^{\varepsilon}\leq Y \leq X$. Suppose that the shift parameters satisfy   $0\leq \Re(z_1) , \Re(z_2) \leq \frac{1}{\log X}$and $|z_1| , |z_2| \leq 4X$.  Then for any positive
 real numbers $\ell_1 , \ell_2, \ell_3$, we have
\begin{multline}
\sum_{X \leq t_{j} \leq X + Y}L(\frac{1}{2},\phi_j)^{\ell_1}L(\frac{1}{2},\Sym^2 \phi \times \phi_j)^{\ell_2}|L(\frac{1}{2}+z_1,\phi_j)L(\frac{1}{2}+z_2 , \phi_j)|^{\ell_3}\\ \ll_{\varepsilon} XY (\log X)^{\frac{\ell_1(\ell_1 -1)}{2}+\frac{\ell_2(\ell_2 -1)}{2} + \varepsilon}\exp(\ell_3(\ell_1 - \frac{1}{2})\mathcal{M}(z_1,z_2 , X) + \frac{\ell_3^2}{8}\mathcal{V}(z_1,z_2 , X)).
\end{multline}
Here $\mathcal{M}(z_1,z_2 , x)$ and $\mathcal{V}(z_1,z_2 , x)$ are defined in (\ref{Mean-value-L-function}), (\ref{Variance-L-function}) respectively.
\end{theorem}
Taking $ z_1 = 0$ and $z_2 = 2iT$, we get $\mathcal{M}(z_1,z_2 , X) = \log\log X$ and $\mathcal{V}(z_1,z_2,X) = 6\log\log X$.  The exponent becomes 
\[
\frac{\ell_1(\ell_1 - 1)}{2}+\frac{\ell_2(\ell_2 - 1)}{2} - \ell_3(\ell_1 - \frac{1}{2})+ \frac{3}{4}\ell_3^2.
\]
Substituting $\ell_1$ by $\ell_1 - \ell_3$, we derive the following corollary.
\begin{corollary}
\label{Mixed-moments-Soundararajan's-method}
Assume GRH and GRC. Let $\phi$ be a Hecke--Maass form with spectral parameter $t_{\phi}$ and suppose $1\leq t_{\phi}^{\delta}\leq T \leq t_{\phi}$ for some $\delta > 0$. Let $\varepsilon > 0$
 be sufficiently small, $T^{1-\varepsilon}
  \leq X\leq 3t_{\phi}$, and $X^{\varepsilon}\leq Y \leq X$. Then for any positive
 real numbers $\ell_1 , \ell_2, \ell_3$ we have
\[
\sum_{X \leq t_{j} \leq X + Y}L(\frac{1}{2},\phi_j)^{\ell_1}L(\frac{1}{2},\Sym^2 \phi \times \phi_j)^{\ell_2}|L(\frac{1}{2}+2iT,\phi_j)|^{\ell_3} \ll_{\varepsilon} XY (\log X)^{\frac{\ell_1(\ell_1 -1)}{2}+\frac{\ell_2(\ell_2 -1)}{2}+\frac{\ell_3^2}{4} + \varepsilon}.
\]
\end{corollary}

\begin{remark}
We need the case  $(\ell_1,\ell_2,\ell_3) = (\frac{3}{2},\frac{1}{2},1)$ to prove Theorem \ref{Nonequidistribution-theorem}. A key observation is that the exponent of $\log X$ in this case is less than $1$. 
\end{remark}

To establish Theorem \ref{equidistribution-phenomenon}, a small shift in the parameters is required. We choose $|\Im (z_1)| \asymp t$ and $|\Im (z_1)| \asymp T \pm t$ with a shift satisfying $|t|\ll t_{\phi}^{\varepsilon}$. Moreover, in this case, we take $(\ell_1,\ell_2,\ell_3) = (\frac{1}{2},\frac{1}{2},1)$ and $(0,0,2)$ in Theorem \ref{Mixed-moments-Soundararajan's-method-shifted} which yields the estimates required for the proof.

\begin{proposition}\label{Mixed-moments-Soundararajan's-method-key-point}
Assume GRH and GRC. Let $\phi$ be a Hecke--Maass form with spectral parameter $t_{\phi}$ and suppose $1\leq t_{\phi}^{\delta}\leq T \leq t_{\phi}$ for some $\delta > 0$. Let $\varepsilon > 0$
 be sufficiently small, $T^{1-\varepsilon}
  \leq X\leq 3t_{\phi}$, and $X^{\varepsilon}\leq Y \leq X$. Let  $|t| \ll t_{\phi}^{\varepsilon}$.  We have
\begin{multline}
    \sum_{X \leq t_{j} \leq X + Y}L(\frac{1}{2},\phi_j)^{1/2}L(\frac{1}{2},\Sym^2 \phi \times \phi_j)^{1/2}|L(\frac{1}{2}+\frac{1}{\log t_{\phi}}\pm it,\phi_j)L(\frac{1}{2}+\frac{1}{\log t_{\phi}}\pm i t + 2iT,\phi_j)|\\ \ll_{\varepsilon} XY (\log (X+t_{\phi}))^{1/2 + \varepsilon}
\end{multline}
and
\begin{multline}
       \sum_{X \leq t_{j} \leq X + Y}|L(\frac{1}{2}+\frac{1}{\log t_{\phi}}\pm it,\phi_j)L(\frac{1}{2}+\frac{1}{\log t_{\phi}}\pm i t + 2iT,\phi_j)|^2\\ \ll_{\varepsilon} XY (\log (X+t_{\phi}))^{2 + \varepsilon} 
\end{multline}
\end{proposition}

We apply Soundararajan’s method in \cite{MR2552116}  to study these
moments under GRH and GRC. Estimating the certain shifted moments also appears in \cite{MR2677611} for a sharp upper bound of second moment of quadratic twists modular $L$-functions.

\subsection{Ideas of proofs}
To prove Theorem \ref{Nonequidistribution-theorem}, it suffices to estimate the sum over the discrete spectrum since the main term has already been established. By inserting absolute values into the sum and applying Watson's formula \cite{watson2008rankintripleproductsquantum} (assume $T \leq t_{\phi}$) as the work in \cite{hua2024jointvaluedistributionheckemaass}, we get
\begin{multline}\label{sum-non-truncated}
    \sum\limits_{j\geq 1}\langle \phi_{j} , \phi^{2}\rangle\langle |E_{T}^{2}|, \phi_{j}\rangle \ll \sum_{j \geq 1}\frac{\Lambda(\frac{1}{2},\phi\times\phi\times\phi_j)^{1/2}\Lambda(\frac{1}{2},E_T\times E_T\times\phi_j)^{1/2}}{\Lambda(1,\Sym^2 \phi_{j})\Lambda(1,\Sym^2 \phi)|\xi(1+2iT)|^2}\\
    \ll \frac{(\log (t_{\phi} + T))^\varepsilon}{t_{\phi}^{1/4}T^{1/4}}\sum_{t_j \leq 2T+t_{\phi}^{\varepsilon}}\frac{L(\frac{1}{2},\phi_j)^{3/2}L(\frac{1}{2},\Sym^2 \phi \times \phi_j)^{1/2}|L(\frac{1}{2}+2iT,\phi_j)|}{|t_j| (1 + |t_j - 2T|)^{1/4}(1 + |t_j - 2t_{\phi}|)^{1/4}}.
\end{multline}
Then the proof of Theorem \ref{Nonequidistribution-theorem} is directly from the mixed moments of $L$-functions. 

For the truncated version Theorem \ref{equidistribution-phenomenon}, we need do more to bound the discrete spectrum from truncated term (see (\ref{J})). By unfolding trick and Mellin inversion, the problem can again be reduced to estimating moments of $L$-functions. After a technical but harmless average (see Proposition \ref{equidistribution-average}), we are led to a more challenging shifted moment estimate, roughly of the form
\begin{multline}\label{sum-truncated}
  \sum_{j \geq 1}\frac{\Lambda(\frac{1}{2},\phi\times\phi\times\phi_j)^{1/2}\Lambda(\frac{1}{2}+z,E_T\times E_T\times\phi_j)^{1/2}}{\Lambda(1,\Sym^2 \phi_{j})\Lambda(1,\Sym^2 \phi)|\zeta(1+2iT)|^2}\\
    \ll \frac{(\log (t_{\phi} + T))^\varepsilon}{t_{\phi}^{1/4}T^{1/4}}\sum_{t_j \leq 2T+t_{\phi}^{\varepsilon}}\frac{L(\frac{1}{2},\phi_j)^{1/2}L(\frac{1}{2},\Sym^2 \phi \times \phi_j)^{1/2}|L(\frac{1}{2}+z,\phi_j)L(\frac{1}{2}+z+2iT,\phi_j)|}{|t_j|(1 +  |t_j - 2T - t|)^{1/4}(1 + |t_j - 2t_{\phi}|)^{1/4}}
\end{multline}
where $z = \frac{1}{\log t_{\phi}}+it$ with $t \ll t_{\phi}^{\varepsilon}$. Clearly, setting $z = 0$ then the sum equals (\ref{sum-non-truncated}).

Thus, the proofs of both theorems are reduced to the study of moments of $L$-functions, which is completed in Theorem \ref{Mixed-moments-Soundararajan's-method-shifted}.

\subsection{Plan for this paper} The rest of this paper is organized as follows. In \S \ref{sec:2}, we review the necessary background on automorphic forms， the triple product formula and regularized formula. In \S \ref{sec:3}, we prove Theorem \ref{Nonequidistribution-theorem} assuming Theorem \ref{Mixed-moments-Soundararajan's-method-shifted} . In \S \ref{sec:4}, we use an extra average and prove the Theorem \ref{equidistribution-phenomenon} under  Theorem \ref{Mixed-moments-Soundararajan's-method-shifted}. In \S \ref{sec:5}, we deduce Theorem \ref{Mixed-moments-Soundararajan's-method-shifted} under GRH and GRC. In Appendix \ref{sec:6}, we give some remarks on the Conjecture \ref{Conjecture-Joint-value-distribution}, especially about $distinguishing$ two Eisenstein series.  In Appendix \ref{sec:7}, we analyze the difference of two mixed fourth moments and calculate the main term (\ref{Main-term}) by another way.

 
\textbf{Notation.} Throughout the paper, $\varepsilon$ is an arbitrarily small positive number; all of them
may be different at each occurrence. As usual, $e(x) = e
^{2\pi ix}$. We use the standard Landau and Vinogradov notations \( O(\cdot) \), \( o(\cdot) \), \(\ll\),  \(\gg\), $\asymp$ and $\sim$. Specifically, we express \( X \ll Y \), \( X = O(Y) \), or \( Y \gg X \) when there exists a constant \( C \) such that \( |X| \leq C|Y| \). If the constant $C=C_s$ depends on some object $s$, we write $X=O_s(Y)$. As \( N \to \infty \), \( X = o(Y) \) indicates that \( |X| \leq c(N)Y \) for some function \( c(N) \) that tends to zero.
We use \( X \asymp Y \) to denote that $c_1Y\leq X\leq c_2Y$ for some positive constant $c_1, c_2$.

\section{\label{sec:2}Preliminaries}
\subsection{Automorphic forms}Let $\{\phi_{k}\}_{k\geq 1}$ be an orthonormal basis of Hecke--Maass cusp forms for $\SL(2,\mathbb{Z})$. We always assume all $\phi_{k}$ are real and normalized by $\int_{\mathbb{X}}\phi_{k}^{2}\dd\mu z = 1$. Denote the spectral parameter of $\phi_{k}$ by $t_{k}$ and the Fourier coefficients(Hecke eigenvalues) $\lambda_{k}(n)$. Also we sometimes write the spectral parameter of a Hecke-Maass form $f$ by $t_{f}$. For a Hecke-Maass form $\phi_{k}$, we have Fourier expansion
\[\phi_{k}(z) = 2\sqrt{y}\rho_{k}(1)\sum\limits_{n \neq 0}\lambda_{k}(n)K_{it_{k}}(2\pi|n|y)e(nx),\]
where
\[|\rho_{k}(1)|^2 = \frac{\cosh\pi t_{k}}{2 L(1,\sym^{2}\phi_{k})}.\]
For Eisenstein series $E(z,s)$, we also have Fourier expansion
\[E(z , s) =  y^{s} + \frac{\xi(2s-1)}{\xi(2s)}y^{1-s} +\frac{ 2\sqrt{y} }{\xi(2s)}\sum\limits_{n\neq0}|n|^{s-\frac{1}{2}}\sigma_{1-2s}(|n|)K_{s-\frac{1}{2}}(2\pi|n|y)e^{2\pi i n x}\]
where $\sigma_{s}(n) := \sum\limits_{ab = n}b^{s}$. We usually write $E(z,1/2+it) = E_{t}(z)$, and we have
\[E(z, 1/2+it) = y^{1/2+it} + \frac{\xi(2it)}{\xi(1+ 2it)}y^{1/2-it} +\frac{ 2\sqrt{y} }{\xi(1+2it)}\sum\limits_{n\neq0}\eta_{t}(|n|)K_{it}(2\pi|n|y)e^{2\pi i n x}\]
where $\eta_{t}(n) = \sum\limits_{ab = n}(\frac{a}{b})^{it}$. Denote $\rho_{t}(1) := 1/\xi(1+2it)$, $\rho_{t}(n) = \rho_{t}(1)\eta_{t}(n)$.\par
\[|\rho_{t}(1)|^{2} = \frac{\cosh\pi t}{|\zeta(1+2it)|^{2}} .\]
By \cite{MR1289494} \cite{MR1067982} and the standard estimate of Riemann zeta function, we have
\[ (\log t_{k})^{-1}\ll L(1,\Sym^{2}\phi_{k}) \ll t_{k}^{\varepsilon}, \quad (\log(1+ |t| ))^{-1} \ll \zeta(1+2it) \ll \log(1+ |t| ).\]
Under GRH and GRC, we have
\[ \frac{1}{\log\log t_{k}}\ll L(1,\Sym^{2}\phi_{k}) \ll (\log\log t_{k})^3, \quad \frac{1}{\log\log(1+ |t| )} \ll \zeta(1+2it) \ll \log\log(1+ |t| ).\]

We introduce  the truncated Eisenstein series. Let
 \[
 e(y,\frac{1}{2}+iT) = y^{\frac{1}{2}+iT} + \frac{\xi(1-2iT)}{\xi(1+2iT)}y^{\frac{1}{2}-iT}.
 \]
We choose $A \geq 1$ fixed and $\mathcal{F}$ the standard fundamental domain of $\mathbb{X}$ that is
\[
\mathcal{F} = \{x+iy \in \mathbb{H} : -\frac{1}{2}\leq x \leq \frac{1}{2}, x^2+y^2 \geq 1 \}.
\]
 Now we define $\mathcal{F}_{A}$, $\mathcal{C}_A$ respectively by
\[
\mathcal{F}_{A} = \{x+iy \in \mathcal{F} : y \leq A\}, \quad \mathcal{C}_{A} = \{x+iy \in \mathcal{F} : y > A\}.
\]
On $\mathcal{F}$, the truncated Eisenstein series is defined by
\begin{equation}\label{truncated-Eisenstein-series}
    \begin{aligned}
  E_{T}^{A}(z) =     \left\{\begin{array}{lr}
          E_{T}(z),   & z \in \mathcal{F}_A,  \\
          E_{T}(z) - e(y, \frac{1}{2}+iT),  & z\in\mathcal{C}_A. 
        \end{array}
        \right.
    \end{aligned}
\end{equation}
We extend the definition to $\mathbb{H}$ by $\SL(2,\mathbb{Z})$ translation. Note that $E_{T}^A(z)$ is an automorphic function.

\subsection{Stirling's formula}
For fixed $\sigma\in\mathbb{R}$, real $|t|\geq10$ and any $J>0$, we have Stirling's formula
\begin{equation}\label{eqn:Stirling_J}
  \Gamma(\sigma+it) = e^{-\frac{\pi}{2}|t|} |t|^{\sigma-\frac{1}{2}} \exp\left( it\log\frac{|t|}{e} \right) \left( g_{\sigma,J}(t) + O_{\sigma,J}(|t|^{-J}) \right),
\end{equation}
where
\[
  t^j \frac{\partial^j}{\partial t^j} g_{\sigma,J}(t) \ll_{j,\sigma,J} 1.
\]
For $s = \sigma +it$, we also have
\[\frac{\Gamma^{\prime}}{\Gamma}(s) = \log s + \mathcal{O}(\frac{1}{|s|}).\]
\subsection{Rankin--Selberg theory and Watson's formula}
Let $\phi , \phi_{k}, \phi_{j}$ be the Hecke--Maass cusp forms and $E_{t}, E_{\tau}$ be the Eisenstein series with spectral parameters $t_{\phi}, t_{k},t_{j}, t , \tau$ respectively.

 By Rankin--Selberg method (see \cite[\S 7.2]{Goldfeld2006AutomorphicFA})  we have
\[\langle \phi_{j}E_{t} , \phi_{k}\rangle =  \frac{\rho_{j}(1)\rho_{k}(1)\Lambda(1/2+it ,\phi_{k}\times \phi_{j})}{\xi(1+2it)},\]
\[\langle E_{\tau}E_{t}, \phi_{k}\rangle = \frac{\rho_{k}(1)\rho_{t}(1)\Lambda(1/2+i\tau + it ,\phi_{k})\Lambda(1/2+i\tau - it ,\phi_{k})}{\xi(1+2i\tau)}.\]
By Watson's formula \cite{watson2008rankintripleproductsquantum}, we have
\[
  |\langle \phi_k \phi,  \phi_j \rangle |^2
  = \frac{\Lambda(1/2,\phi_k\times \phi\times \phi_j)}{8 \Lambda(1,\Sym^2 \phi_k)\Lambda(1,\Sym^2 \phi)\Lambda(1,\Sym^2 \phi_j)}
\]
and
\[
  |\langle \phi_j ,\phi^2  \rangle|^2
  = \frac{\Lambda(1/2,\phi_j) \Lambda(1/2, \Sym^2 \phi\times \phi_j)}{8  \Lambda(1,\Sym^2 \phi)^2 \Lambda(1,\Sym^2 \phi_j)}.
\]
\subsection{Regularized inner product and regularized Plancherel formula}To treating the inner product of Eisenstein series, we will make use of the regularization process given by Zagier in \cite{zagier1981rankin}. 

Let $F(z)$ be a continuous $\SL(2,\mathbb{Z})$-invariant function on $\mathbb{H}$. It is called \emph{renormalizable} if there is a function $\Phi(y)$ on $\mathbb{R}_{>0}$ of the form
\begin{equation} \label{Phi_def}
\Phi(y)=\sum_{j=1}^l \frac{c_j}{n_j!} y^{\alpha_j} \log^{n_j} y,
\end{equation}
with $c_j, \alpha_j \in \mathbb{C}$ and $n_j \in \mathbb{Z}_{\ge 0}$, such that
$$
F(z)= \Phi(y) + O(y^{-N})
$$
as $y \rightarrow \infty$, and for any $N>0$.

If $F(z)=\sum_{n= - \infty}^{\infty} a_n(y) e(n x)$ is the Fourier expansion of $F$ at the cusp $\infty$, in particular if $a_0(y)$ is its 0-term, and if no $\alpha_j$ equals 0 or 1, then the function
$$
R(F, s):=\int_0^{\infty} (a_0(y) - \Phi(y) )  y^{s-2} dy,
$$
where the defining integral converges for sufficiently large ${\rm Re}(s)$, can be meromorphically continued to all $s$ and has a simple pole at $s=1$. Then one can define the regularized integral with
\begin{equation} \label{reg_first_by_R(F,s)}
\int_{\mathbb{X}}^{reg} F(z) d\mu(z) := \frac{\pi}{3}  {\rm Res}_{s=1} R(F, s).
\end{equation}

Under the assumption that no $\alpha_j=1$, let $\mathcal{E}_{\Phi}(z)$ denote a linear combination of Eisenstein series $E(z, \alpha_j)$ (or suitable derivatives thereof) corresponding to all the exponents in (\ref{Phi_def}) with ${\rm Re}(\alpha_j) > 1/2$, i.e. such that $F(z) - \mathcal{E}_{\Phi}(z)=O(y^{1/2})$. An important definition of regularization is given by
\begin{equation} \label{reg_subtract_Eisen}
\int_{\mathbb{X}}^{reg} F(z) \dd\mu z
=\int_{\mathbb{X}} (F(z) -\mathcal{E}_{\Phi}(z)) \dd\mu z.
\end{equation}
The triple product formula for Eisenstein series is 
\begin{lemma}[\cite{zagier1981rankin}]\label{triple-product-Eisenstein-series}

    \begin{equation} \label{3eisen}
    \begin{aligned}
   \int_{\mathbb{X}}^{reg}& E(z, 1/2 + s_1) E(z, 1/2 + s_2) E(z, 1/2 + s_3)  \dd\mu z   \\
   &=\frac{\xi(1/2 +s_1 +s_2 +s_3) \xi(1/2 +s_1 -s_2 +s_3) \xi(1/2+s_1 +s_2 -s_3) \xi(1/2+s_1 -s_2 -s_3)}{\xi(1+ 2 s_1) \xi(1+ 2 s_2) \xi(1+ 2 s_3)}.   
    \end{aligned}
\end{equation}
\end{lemma}
The regularized Plancherel formula in classical language is mentioned in \cite{MR4184616}.
\begin{lemma}\label{Prop_Regular_Planch} Let $F(z)$ and $G(z)$ be renormalizable functions on $\Gamma \backslash \mathbb{H}$ such that $F - \Phi$ and $G - \Psi$ are of rapid decay as $y \rightarrow \infty$, for some $\Phi(y)=\sum_{j=1}^l \frac{c_j}{n_j!} y^{\alpha_j} \log^{n_j} y$ and $\Psi(y)=\sum_{k=1}^{l_1} \frac{d_k}{m_k!} y^{\beta_k} \log^{m_k}y$. Moreover, let $\alpha_j \neq 1$, $\beta_k \neq 1$, ${\rm Re}(\alpha_j) \neq 1/2$, ${\rm Re}(\beta_k) \neq 1/2$, $\alpha_j + \overline{\beta_k} \neq 1$ and $\alpha_j \neq \overline{\beta_k}$, for all $j, k$. Then the following formula holds:
\begin{equation*}
    \begin{aligned}
\langle F(z), G(z) \rangle_{reg}
=&\langle F, \sqrt{3/ \pi}  \rangle_{reg} \langle  \sqrt{3/ \pi} , G  \rangle_{reg} +   \sum_j \langle F , u_j \rangle \langle  u_j, G \rangle  \\
&+ \frac{1}{4 \pi} \int_{-\infty}^{\infty} \langle F, E_{t}\rangle_{reg}  \langle  E_{t}, G \rangle_{reg} \dd t
+ \langle F, \mathcal{E}_{\Psi}  \rangle_{reg}  +  \langle  \mathcal{E}_{\Phi}, G \rangle_{reg}.
    \end{aligned}
\end{equation*}
\end{lemma}

\section{\label{sec:3} The proof of Theorem \ref{Nonequidistribution-theorem}}In this section, we prove Theorem \ref{Nonequidistribution-theorem} by assuming Corollary \ref{Mixed-moments-Soundararajan's-method}.

Recall that
\[
\langle \phi^2 , |E_{T}|^2\rangle = \mathcal{R}(\phi,E_T)
+ \sum\limits_{j\geq 1}\langle \phi_{j} , \phi^{2}\rangle\langle |E_{T}^{2}|, \phi_{j}\rangle + \frac{1}{4\pi}\int_{\mathbb{R}}\langle \phi^{2}, E_{\tau}\rangle\langle E_{\tau},|E_{T}^{2}|\rangle_{reg}\dd \tau 
\]
where
\begin{multline}
   \mathcal{R}(\phi, E_T) = \frac{3}{\pi}[\log (\frac{1}{4} + t_{\phi}^2) + \log (\frac{1}{4} + T^2) + 2\frac{L^{\prime}}{L}(1,\Sym^2 \phi)+4\Re\frac{\zeta'}{\zeta}(1+2iT)] + \mathcal{O}(1)\\
+ \frac{\xi(2iT)}{\xi(1+2iT)}\langle \phi^2 , E(z,1+2iT)\rangle + \overline{\frac{\xi(2iT)}{\xi(1+2iT)}}\langle \phi^2 , E(z,1-2iT)\rangle. 
\end{multline}

\subsection{The main term}
Under GRH and GRC, we get
\[
 \mathcal{R}(\phi, E_T) = \frac{3}{\pi}[\log (\frac{1}{4} + t_{\phi}^2) + \log (\frac{1}{4} + T^2)] + \mathcal{O}((\log (T + t_{\phi}))^{\varepsilon}).
\]
\subsection{The contribution of continuous spectrum}
By Rankin--Selberg method and Zagier's formula (\ref{triple-product-Eisenstein-series})
we get the continuous term is
\[
\int_{|\tau| \leq 2T+t_{\phi}^{\varepsilon}}\frac{\Lambda(\frac{1}{2}-i\tau , \Sym^2 \phi)\xi(\frac{1}{2}-i\tau)}{\Lambda(1,\Sym^2\phi)\xi(1-2i\tau)}\frac{\xi(\frac{1}{2}+i\tau +2iT)\xi(\frac{1}{2}+i\tau -2iT)\xi(\frac{1}{2}+i\tau )^2}{\xi(1+2i\tau)|\xi(1+2iT)|^2} \dd \tau.
\]
After a direct calculation of Gamma factor, we can bound it by
\begin{equation}\label{continous-spectrum}
    \frac{t_{\phi}^{\varepsilon}}{T^{1/4}t_{\phi}^{1/4}}\int_{0 < \tau \leq 2T+t_{\phi}^{\varepsilon}}\frac{|L(\frac{1}{2}-i\tau,\Sym^2 \phi)|\zeta(\frac{1}{2}+i\tau)|^3 |\zeta(\frac{1}{2}+i\tau +2iT)\zeta(\frac{1}{2}+i\tau -2iT)|}{(1 + |\tau|)(1 + |\tau - 2T|)^{1/4}(1 + |\tau - 2t_{\phi}|)^{1/4}}\dd \tau.
\end{equation}
Then by Generalized Lindel\"of Hypothesis (GLH) we get the contribution of continuous spectrum is $\mathcal{O}(t_{\phi}^{-1/4+\varepsilon})$.


\subsection{The contribution of discrete spectrum} By Watson's formula and a rapid calculation of Gamma factor, we need deal with
\begin{equation}
     \frac{(\log(t_{\phi} + T))^\varepsilon}{t_{\phi}^{1/4}T^{1/4}}\sum_{t_j \leq 2T+t_{\phi}^{\varepsilon}}\frac{L(\frac{1}{2},\phi_j)^{3/2}L(\frac{1}{2},\Sym^2 \phi \times \phi_j)^{1/2}|L(\frac{1}{2}+2iT,\phi_j)|}{|t_j| (1 + |t_j - 2T|)^{1/4}(1 + |t_j - 2t_{\phi}|)^{1/4}}
\end{equation}
under GRH and GRC. We divide the sum to four parts
\[
\sum_{t_j \leq T^{1-\varepsilon}}+\sum_{T^{1-\varepsilon}  \leq t_j \leq T + T^{1-\varepsilon}}+\sum_{T + T^{1-\varepsilon} \leq t_j \leq 2T-T^{\varepsilon}}+\sum_{2T - T^{\varepsilon}\leq t_j \leq 2T+t_{\phi}^{\varepsilon}}.
\]
For the first and the final one we use the sharp upper bound of central value and get $\mathcal{O}(t_{\phi}^{-\varepsilon})$. For the second term we have
\begin{multline}
     \sum_{1\leq k \leq T^{1-\varepsilon}}\sum_{T^{1-\varepsilon}+ (k-1)T^{\varepsilon}\leq t_j \leq    T^{1-\varepsilon} +kT^{\varepsilon}}\frac{L(\frac{1}{2},\phi_j)^{3/2}L(\frac{1}{2},\Sym^2 \phi \times \phi_j)^{1/2}|L(\frac{1}{2}+2iT,\phi_j)|}{|t_j|t_{\phi}^{1/2}T^{1/2}}\\
     \ll \frac{T^{1/2}(\log (t_{\phi}+T))^{1/2+\varepsilon}}{t_{\phi}^{1/2}}
\end{multline}
by Corollary \ref{Mixed-moments-Soundararajan's-method} with $(\ell_1,\ell_2,\ell_3) = (\frac{3}{2},\frac{1}{2},1)$ and partial summation. For the third  term, we bound it by
\begin{multline}\label{der}
\frac{(\log (t_{\phi} + T))^\varepsilon}{T^{1/4}t_{\phi}^{1/4}}\sum_{1\leq k \leq T^{1-\varepsilon}}\sum_{ -(k+1)T^{\varepsilon}\leq t_j -2T \leq -kT^{\varepsilon}}  \frac{\mathcal{L}(t_j)}{t_j(1 + (k-1)T^{\varepsilon})^{1/4}(1 + 2t_{\phi} - 2T + kT^{\varepsilon} )^{1/4}}\\
 \ll \frac{T^{\varepsilon}\log (T+t_{\phi})^{1/2+\varepsilon}}{T^{1/4}t_{\phi}^{1/4}}\sum_{10\leq k \leq T^{1-\varepsilon}}\frac{1}{(kT^{\varepsilon})^{1/4}\max\{2t_{\phi} - 2T +1 , kT^{\varepsilon}\}^{1/4}}.
\end{multline}
Here we use $a+b \geq \max\{a , b\}$ for $a,b \geq 0$.

For $ 2t_{\phi} - 2T + 1\leq T$, we bound it
\begin{multline*}
         \ll \frac{T^{\varepsilon}\log (T+t_{\phi})^{1/2+\varepsilon}}{T^{1/4}t_{\phi}^{1/4}}\sum_{10\leq k \leq \frac{2t_{\phi}-2T + 1}{T^{\varepsilon}}}+\sum_{\frac{2t_{\phi}-2T + 1}{T^{\varepsilon}}\leq k \leq T^{1-\varepsilon} }\frac{1}{(kT^{\varepsilon})^{1/4}\max\{2t_{\phi} - 2T +1 , kT^{\varepsilon}\}^{1/4}}\\
     \ll \frac{\log (T+t_{\phi})^{1/2+\varepsilon}}{T^{1/4}t_{\phi}^{1/4}}((2t_{\phi}  -2T +1)^{1/2} + T^{1/2})
     \ll \frac{T^{1/4}\log (T+t_{\phi})^{1/2+\varepsilon}}{t_{\phi}^{1/4}}.
\end{multline*}

And similarly the term is bounded by
\[
 \frac{\log (T+t_{\phi})^{1/2+\varepsilon}}{T^{1/4}t_{\phi}^{1/4}} \frac{T^{3/4}}{(2t_{\phi} - 2T +1)^{1/4}} \ll \frac{T^{1/4}\log (T+t_{\phi})^{1/2+\varepsilon}}{t_{\phi}^{1/4}}
\]
for $2t_{\phi} - 2T + 1\geq T$.

Then we complete the proof of Theorem \ref{Nonequidistribution-theorem}.

\section{\label{sec:4}The proof of Theorem \ref{equidistribution-phenomenon}}
In this section we prove Theorem \ref{equidistribution-phenomenon}. By spectral decomposition, we get
\begin{equation}
    \begin{aligned}
        \langle \phi^2 ,|E_{T}^{A}|^2\rangle &= \overline{c_T}^2 \langle \phi^2 ,(E_{T}^{A})^2\rangle\\&=\overline{c_T}^2[\frac{3}{\pi}\langle (E_{T}^A)^2 , 1 \rangle + \sum_{j}\langle \phi^2 , \phi_j\rangle\langle \phi_j , (E_{T}^A)^2\rangle + \frac{1}{4\pi}\int_{\mathbb{R}}\langle \phi^2, E_{\tau}\rangle \langle E_{\tau} , (E_{T}^A)^2\rangle \dd \tau]\\
        & = \mathcal{I} + \mathcal{J} + \mathcal{K}.
    \end{aligned}
\end{equation}
The main term is completely from $\mathcal{I}$ (see (\ref{I})). Our main work is to bound $\mathcal{J}$ and $\mathcal{K}$. Before estimating them, we use a smooth average over $A$.
\subsection{A smooth version}
In fact, we prove a version of Theorem \ref{Mixed-moments-Soundararajan's-method-shifted} with an extra average of $A$ as \cite{djanković2024fourthmomenttruncatedeisenstein}. This average is harmless. 

\begin{proposition}\label{equidistribution-average}Let $B$ be a fixed positive number and $h(A)$ a smooth, non-negative real with compactly support in $ B-(\log t_{\phi})^{-10} \leq A \leq B+(\log t_{\phi})^{-10}$. The function $h(A)$ satisfies
\[
(\log t_{\phi})^{-10}\ll \widetilde{h}(1)  = \int_{-\infty}^{\infty}h(A)\dd A \ll (\log t_{\phi})^{-10}
\]
and
\[
 h^{(k)}(A) \ll (\log t_{\phi})^{10k}
\]
where $\widetilde{h}(s) : =\int_{0}^{\infty}h(A)A^{s-1}\dd A$ is the Mellin transform and $k$ is any nonnegative integer. Assume $1\leq t_{\phi}^{\delta} \leq T \leq t_{\phi}$ for a small $\delta > 0 $. Assume GRH and GRC, we get
\begin{equation}\label{average-equation}
\int_{-\infty}^{\infty}h(A)\langle\phi^2(z) , |E_{T}^{A}(z)|^2\rangle\dd A = \frac{3\widetilde{h}(1)}{\pi}\log (\frac{1}{4}+T^2) + \mathcal{O}_{B,\varepsilon}(\widetilde{h}(1)(\log T)^{1/2+\varepsilon})    
\end{equation}
as $\min\{t_{\phi} , T\}$ goes to infinite.
\end{proposition}
\begin{lemma}Under GLH, (\ref{average-equation})
  implies  Theorem \ref{equidistribution-phenomenon}. Hence Proposition \ref{equidistribution-average} implies Theorem \ref{equidistribution-phenomenon}.
\end{lemma}
\begin{proof}Let $B >1$ fixed. We consider 
\[
\widetilde{h}(0)\langle\phi^2(z) , |E_{T}^{B}(z)|^2\rangle = \int_{-\infty}^{\infty}h(A)\langle\phi^2(z) , |E_{T}^{B}(z)|^2\rangle\dd A = I_1 + I_2 +I_3
\]
    where
    \begin{equation}
        \begin{aligned}
I_{1} &=    \int_{-\infty}^{\infty}h(A)\int\limits_{\substack{z \in \mathcal{F}\\ \Im(z) \leq B - (\log t_{\phi})^{-10}}}\phi^2(z) |E_{T}^{A}(z) + (E_{T}^{B}(z)-E_{T}^{A}(z))|^2\dd A,\\    
I_{2} &=    \int_{-\infty}^{\infty}h(A)\int\limits_{\substack{z \in \mathcal{F}\\ B - (\log t_{\phi})^{-10} \leq \Im(z) \leq B + (\log t_{\phi})^{-10}}}\phi^2(z) |E_{T}^{A}(z) + (E_{T}^{B}(z)-E_{T}^{A}(z))|^2\dd A,\\
I_{3} &=    \int_{-\infty}^{\infty}h(A)\int\limits_{\substack{z \in \mathcal{F}\\ \Im(z) \geq B + (\log t_{\phi})^{-10}}}\phi^2(z) |E_{T}^{A}(z) + (E_{T}^{B}(z)-E_{T}^{A}(z))|^2\dd A.
        \end{aligned}
    \end{equation}
Since $A \in [B -  (\log t_{\phi})^{-10}, B + (\log t_{\phi})^{-10}]$, we get 
\[
E_{T}^{B}(z) - E_{T}^{A}(z) = 0
\]
if $ \Im(z) \leq B - (\log t_{\phi})^{-10}$ or $\Im(z) \geq B + (\log t_{\phi})^{-10}$. When $ B - (\log t_{\phi})^{-10} \leq \Im(z) \leq B + (\log t_{\phi})^{-10}$, we have
\[
|E_{T}^{B}(z) - E_{T}^{A}(z)| \leq |y^{1/2-iT} + \frac{\xi(1+2iT)}{\xi(1-2iT)y^{1/2+iT}}| \leq 2y^{1/2} \leq 3B.
\]
Then we get
\begin{multline}\label{differ}
    |\widetilde{h}(1)\langle\phi^2(z) , |E_{T}^{B}(z)|^2\rangle  - \int_{-\infty}^{\infty}h(A)\langle\phi^2(z) , |E_{T}^{A}(z)|^2\rangle\dd A|\\
\ll \int_{-\infty}^{\infty}h(A)\int\limits_{\substack{z \in \mathcal{F}\\ B - (\log t_{\phi})^{-10} \leq \Im(z) \leq B + (\log t_{\phi})^{-10}}}\phi^2(z)\dd \mu z \times 9B^2\dd A\\
+ \int_{-\infty}^{\infty}h(A)\int\limits_{\substack{z \in \mathcal{F}\\ B - (\log t_{\phi})^{-10} \leq \Im(z) \leq B + (\log t_{\phi})^{-10}}}\phi^2(z) |E_{T}^{A}(z)|\dd \mu z \times 3B\dd A.
\end{multline}
Let  $\Omega = \{z \in \mathcal{F} : B - (\log t_{\phi})^{-10} \leq \Im(z) \leq B + (\log t_{\phi})^{-10} \}$ and note that $\vol(\Omega) \ll_{B} (\log t_{\phi})^{-10}$. We have $\|\phi\|_{4}^{4} \ll 1$ and 
\begin{equation}\label{QUE-Shrinking}
    \int\limits_{\substack{z \in \mathcal{F}\\ B - (\log t_{\phi})^{-10} \leq \Im(z) \leq B + (\log t_{\phi})^{-10}}}  |E_{T}^{A}(z)|^2\dd \mu z \ll_B \vol(\Omega)\log T.
\end{equation}
which deduces from QUE of Eisenstein series.  We now prove (\ref{QUE-Shrinking}) directly. By the definition and Fourier expansion, we need only prove
\[
\int_{\Omega}\frac{4y}{|\xi(1+2iT)|^2}\Big|\sum_{n\neq 0}\eta_{t}(|n|)K_{iT}(2\pi|n|y)e(nx)\Big|^2\frac{\dd x\dd y}{y^2} \ll_{B} \vol(\Omega)\log T .
\]
Opening the square and using Mellin inversion, we get the left side becomes 
\begin{equation*}
    \begin{aligned}
&\frac{8}{|\xi(1+2iT)|^2}\int_{B-(\log t_{\phi})^{-10}}^{B+(\log t_{\phi})^{-10}}\sum_{n\geq 1}\eta_{t}(n)^2K_{iT}(2\pi ny)^2\frac{\dd y}{y}\\
&= \frac{1}{|\xi(1+2iT)|^2} \int_{B-(\log t_{\phi})^{-10}}^{B+(\log t_{\phi})^{-10}}\int_{(3)}\frac{\xi(s)^2\xi(s+2iT)\xi(s-2iT)}{\xi(2s)}y^{-s}\dd s\frac{\dd y}{y}\\
& = \frac{1}{|\xi(1+2iT)|^2}\int_{(3)}\frac{\xi(s)^2\xi(s+2iT)\xi(s-2iT)}{\xi(2s)}\times\\ 
&\quad\quad \left(-\frac{(B-(\log t_{\phi})^{-10})^{-s}}{s} + \frac{(B+(\log t_{\phi})^{-10})^{-s}}{s}\right)   \dd s.
    \end{aligned}
\end{equation*}
Shifting contours to $\Re(s) = \frac{1}{2}$ and we take the residue at $s=1,s=1+2iT,s=1-2iT$. Note that 
\[
\left(-\frac{(B-(\log t_{\phi})^{-10})^{-s}}{s} + \frac{(B+(\log t_{\phi})^{-10})^{-s}}{s}\right)\Big|_{s=1} = \vol(\Omega)
\]
and $\log T$ comes from $\frac{\xi'}{\xi}(1+2iT)$. The integral on the line $\Re(s) = \frac{1}{2}$ is bounded by
\[
\int_{|t| \ll T}\frac{|\zeta(\frac{1}{2}+it)^2\zeta(\frac{1}{2}+it+2iT)\zeta(\frac{1}{2}+it-2iT)|}{(1+|t|)^{3/2}(1+|t+2T|)^{1/4}(1+|t-2T|)^{1/4}}\dd t \ll T^{-1/2+\varepsilon}.
\]
By Cauchy--Schwarz inequality and (\ref{QUE-Shrinking}), we deduce 
\[
 |\widetilde{h}(1)\langle\phi^2(z) , |E_{T}^{B}(z)|^2\rangle  - \int_{-\infty}^{\infty}h(A)\langle\phi^2(z) , |E_{T}^{A}(z)|^2\rangle\dd A| \ll_{B} \widetilde{h}(1)(\log t_{\phi})^{-1}
\]
from (\ref{differ}). Then by Proposition \ref{equidistribution-average}, we get
\[
\langle\phi^2(z) , |E_{T}^{B}(z)|^2\rangle  - \frac{3}{\pi}\log (\frac{1}{4}+T^2) \ll_B (\log T)^{1/2+\varepsilon}.
\]

\end{proof}
\subsection{The contribution of constant term $\mathcal{I}$}For the inner product of two truncated Eisenstein series, we have
\[
 \int_{\mathbb{X}}|E_{T}^{A}(z)|^2 \dd \mu z = \lim\limits_{r \rightarrow 0}\int_{\mathbb{X}}E^{A}(z,\frac{1}{2}+iT)\overline{E^A(z,\frac{1}{2}+iT +ir)} \dd \mu z.
\]
Let $\varphi(s) = \frac{\xi(2s-1)}{\xi(2s)}$. We use Maass--Selberg relation \cite[Proposition 6.8]{MR1942691} and get
\begin{equation}
    \begin{aligned}
        \int_{\mathbb{X}}&E^{A}(z,\frac{1}{2}+iT)\overline{E^A(z,\frac{1}{2}+iT +ir)} \dd \mu z\\
    & = \frac{1}{i(2T +r)}\varphi(\frac{1}{2}-iT-ir)A^{i(2T+r)} -\frac{1}{i(2T+r)}\varphi(\frac{1}{2}+iT)A^{-i(2T+r)}\\
    & \quad- \frac{1}{ir}A^{-ir} +\frac{1}{ir}\varphi(\frac{1}{2}+iT) \varphi(\frac{1}{2}-iT-ir)A^{ir}.
    \end{aligned}
\end{equation}
Note that the Taylor expansion
\[
A^{\pm ir} = 1 \pm ir\log A +\mathcal{O}((r\log A)^2) ,
\]
\[
\varphi(\frac{1}{2}-iT-ir) = \varphi(\frac{1}{2}-iT) -ir\varphi'(\frac{1}{2}-iT) + \mathcal{O}(r^2)
\]
and
\[
|\varphi(\frac{1}{2}+it)| = 1 
\]
for any $t \neq 0$. Then we have the first two terms are bounded by $\mathcal{O}(\frac{1}{T})$. The final two terms are
\[
2\log A -\varphi(\frac{1}{2}+iT)\varphi'(\frac{1}{2}-iT) + \mathcal{O}(r).
\]
By definition, we get
\begin{equation}
    \varphi(\frac{1}{2}+iT)\varphi'(\frac{1}{2}-iT)  = -2\frac{\xi'}{\xi}(1+2iT)-2\frac{\xi'}{\xi}(1-2iT) = -4\Re\frac{\xi'}{\xi}(1+2iT).
\end{equation}
By Stirling's formula and let $r \rightarrow 0$, we have
\begin{equation}\label{I}
    \mathcal{I} = \frac{3}{\pi}\Big(2\log A + 4\Re\frac{\xi'}{\xi}(1+2iT)\Big) = \frac{3}{\pi}\log\left((\frac{1}{4}+T^2)A^2\right) + \mathcal{O}((\log T)^{\varepsilon})
\end{equation}
under RH.
\subsection{The contribution of continuous spectrum $\mathcal{K}$}From Spinu's work \cite{MR2705106} we have
\[
\frac{1}{4\pi}\int_{\mathbb{R}}|\langle (E_{T}^{A})^2, E_{\tau}\rangle|^2 \dd \tau \leq 108A + \mathcal{O}(T^{-\delta}).
\]

By Watson's formula and GLH, we have
\[
\frac{1}{4\pi}\int_{\mathbb{R}}|\langle \phi^2, E_{\tau}\rangle|^2 \dd \tau \ll t_{\phi}^{-1/2+\varepsilon}.
\]
By Cauchy-Schwarz inequality, we get
\begin{equation}\label{K}
    \mathcal{K} \ll A^{1/2}t_{\phi}^{-1/4+\varepsilon}.
\end{equation}

\subsection{The contribution of discrete spectrum $\mathcal{J}$}

We write
\begin{equation}\label{J}
    \begin{aligned}
        \mathcal{J} &= \sum_{j}\langle \phi^2 , \phi_j\rangle\langle \phi_j , |E_{T}|^2\rangle - \sum_{j}\langle \phi^2 , \phi_j\rangle\langle \phi_j , |E_T|^2 - |E_{T}^A|^2\rangle\\
        & = \sum_{j}\langle \phi^2 , \phi_j\rangle\langle \phi_j , |E_{T}|^2\rangle - \overline{c_T}^2\sum_{j}\langle \phi^2 , \phi_j\rangle \int_{\mathcal{C}_A}\phi_j(z) \overline{e(y,\frac{1}{2}+iT)(e(y,\frac{1}{2}+it)+2E_{T}^A(z)}\dd\mu z\\
        & = \sum_{j}\langle \phi^2 , \phi_j\rangle\langle \phi_j , |E_{T}|^2\rangle - \sum_{j}\langle \phi^2 , \phi_j\rangle\times 2\int_{\mathcal{C}_A}\phi_j(z) (y^{1/2-iT}+\frac{\xi(1+2iT)}{\xi(1-2iT)}y^{1/2+iT})E_{T}^{A}(z)\dd\mu z.
    \end{aligned}
\end{equation}

We write
\[
\mathcal{J}(A):=\sum_{j}\langle \phi^2 , \phi_j\rangle\times 2\int_{\mathcal{C}_A}\phi_j(z) (y^{1/2-iT}+\frac{\xi(1+2iT)}{\xi(1-2iT)}y^{1/2+iT})E_{T}^{A}(z)\dd\mu z.
\]
Our main goal in this subsection is
\begin{proposition}\label{J_A-average}Under GRH and GRC, we have
    \[
     \int_{-\infty}^{+\infty}h(A) \mathcal{J}(A) \dd A \ll_{B,\varepsilon} \widetilde{h}(1)(\log T)^{1/2+\varepsilon}.
    \]
\end{proposition}
We need calculate the integral with type
\[
\int_{A}^{\infty}\int_{-1/2}^{1/2}\phi_j(z)y^{1/2-iT}E_{T}^{A}(z) \frac{\dd x\dd y}{y^2}.
\]
By Fourier expansion, we get
\begin{equation}\label{integral-A-truncate}
    8\rho_{j}(1)\rho_t(1) \sum_{n\geq 1}\lambda_{j}(n)\eta_{T}(n) \int_{A}^{\infty}K_{it_{j}}(2\pi ny)K_{iT}(2\pi ny) y^{1/2+iT} \frac{\dd y}{y}.
\end{equation}
There is a useful lemma about the Mellin transform of $K$-Bessel functions.
\begin{lemma}\label{Mellin-Barnes-formula} Let $\Re(s)$ be sufficiently large, we have
\[
\int_{0}^{\infty}K_{\mu}(x)K_{\nu}(x)x^{s}\frac{\dd x}{x} = 2^{s-3}\frac{\prod_{\pm_1}\prod_{\pm_2 }\Gamma(\frac{s \pm_1 \mu \pm_2 \nu}{2})}{\Gamma(s)}.
\]
\end{lemma}
By Lemma \ref{Mellin-Barnes-formula} and Mellin inversion, we reduce (\ref{integral-A-truncate}) to
\[
\rho_{j}(1)\rho_t(1) \frac{1}{2\pi i }\int_{(\sigma)}\pi^{-s}\frac{\prod_{\pm_1}\prod_{\pm_2}\Gamma(\frac{s\pm_1 it_j \pm_2 iT}{2})}{\Gamma(s)}\sum_{n \geq 1}\frac{\lambda_{j}(n)\eta_{T}(n)}{n^{s}}\frac{A^{1/2-s - iT}}{s-1/2 + iT} \dd s.
\]
By Ramanujan identity, we get
\[
\sum_{n \geq 1}\frac{\lambda_{j}(n)\eta_{T}(n)}{n^{s}} = \frac{L(s - iT, \phi_j)L(s + iT,\phi_j)}{\zeta(2s)}.
\]
Then we rewrite the integral by
\begin{equation}
    \begin{aligned}
        \frac{1}{2\pi i }&\int_{(\sigma)}\frac{\rho_j(1)\rho_T(1)\Lambda(s-iT,\phi_j)\Lambda(s+iT,\phi_j)}{\xi(2s)}\frac{A^{1/2-s - iT}}{s-1/2 + iT} \dd s\\
        = &\frac{1}{2\pi i }\int_{(\sigma)}\frac{\rho_j(1)\rho_T(1)\Lambda(s+\frac{1}{2},\phi_j)\Lambda(s +\frac{1}{2}-2iT,\phi_j)}{\xi(1 + 2s - 2iT)}\frac{A^{-s}}{s} \dd s\\
        & = \frac{1}{2\pi i}\int_{(\sigma)}\frac{L(s+\frac{1}{2},\phi_j)L(s +\frac{1}{2}-2iT,\phi_j)}{\zeta(1 + 2iT)\zeta(1 + 2s - 2iT)}\frac{\prod_{\pm}\Gamma(\frac{s+\frac{1}{2}\pm it_j}{2})\Gamma(\frac{s+\frac{1}{2} - 2i T\pm it_j}{2})}{|\Gamma(\frac{1}{2}+it_j)|\Gamma(\frac{1}{2}+iT)\Gamma(\frac{1}{2}+s -iT)}\frac{A^{-s}}{s} \dd s.
    \end{aligned}
\end{equation}
Another part is the conjugation of the above so we only consider this term.
Let $h(A)$ be a smooth function with compact support in $B-(\log t_\phi)^{10} \leq A \leq B+(\log t_{\phi})^{10}$. We get
\begin{multline}
    \int_{-\infty}^{+\infty}h(A) \mathcal{J}(A) \dd A = \sum_{j}\langle \phi^2 , \phi_j\rangle\times\\
     \frac{1}{2\pi i}\int_{(\sigma)}\widetilde{h}(1-s)\frac{L(s+\frac{1}{2},\phi_j)L(s +\frac{1}{2}-2iT,\phi_j)}{\zeta(1 + 2iT)\zeta(1 + 2s - 2iT)}\frac{\prod_{\pm}\Gamma(\frac{s+\frac{1}{2}\pm it_j}{2})\Gamma(\frac{s+\frac{1}{2} - 2i T\pm it_j}{2})}{|\Gamma(\frac{1}{2}+it_j)|\Gamma(\frac{1}{2}+iT)\Gamma(\frac{1}{2}+s -iT)}\frac{\dd s}{s} 
\end{multline}
where
\[
\widetilde{h}(s) = \int_{-\infty}^{\infty}h(A)A^{s}\frac{\dd A}{A}.
\]
By Stirling's formula and $\widetilde{h}(s) \ll \left(\frac{(\log t_{\phi})^{10}}{1+|s|}\right)^{A}$ for any $A > 0$, we shift the the contour of integration to $\Re(s) = \frac{1}{\log t_{\phi}}$. We can truncate the integral to $ |t| \leq t_{\phi}^{\varepsilon}$ with a negligible power saving error term. Now the remaining portion is divided to two parts $(\log t_{\phi})^{20}\leq |t| \leq t_{\phi}^{\varepsilon}$ and $|t |\leq (\log t_{\phi})^{20}$. We define these two parts  $\mathcal{J}_1(B)$ and $\mathcal{J}_2(B)$ by
\begin{multline}
    \mathcal{J}_1(B)  = \frac{1}{2\pi i}\sum_{j}\langle \phi^2 , \phi_j\rangle\int\limits_{(\log t_{\phi})^{20}\leq|t| \leq t_{\phi}^{\varepsilon}}\widetilde{h}(1-s) \times\\
     \frac{L(s+\frac{1}{2},\phi_j)L(s +\frac{1}{2}-2iT,\phi_j)}{\zeta(1 + 2iT)\zeta(1 + 2s - 2iT)}\frac{\prod_{\pm}\Gamma(\frac{s+\frac{1}{2}\pm it_j}{2})\Gamma(\frac{s+\frac{1}{2} - 2i T\pm it_j}{2})}{|\Gamma(\frac{1}{2}+it_j)|\Gamma(\frac{1}{2}+iT)\Gamma(\frac{1}{2}+s -iT)}\frac{\dd s}{s}, 
\end{multline}
\begin{multline}
    \mathcal{J}_2(B) =  \frac{1}{2\pi i}\sum_{j}\langle \phi^2 , \phi_j\rangle\int\limits_{|t|\leq(\log t_{\phi})^{20}}\widetilde{h}(1-s)\times\\
    \frac{L(s+\frac{1}{2},\phi_j)L(s +\frac{1}{2}-2iT,\phi_j)}{\zeta(1 + 2iT)\zeta(1 + 2s - 2iT)}\frac{\prod_{\pm}\Gamma(\frac{s+\frac{1}{2}\pm it_j}{2})\Gamma(\frac{s+\frac{1}{2} - 2i T\pm it_j}{2})}{|\Gamma(\frac{1}{2}+it_j)|\Gamma(\frac{1}{2}+iT)\Gamma(\frac{1}{2}+s -iT)}\frac{\dd s}{s} .
\end{multline}
Thus, We have 
\[
  \int_{-\infty}^{+\infty}h(A) \mathcal{J}(A) \dd A = \mathcal{J}_1(B) + \mathcal{J}_2(B) + \mathcal{O}(t_{\phi}^{-100}).
\]
The goal of the remainder of this section is to establish the bounds  
\[
\mathcal{J}_1(B),\quad \mathcal{J}_2(B) \ll_B \widetilde{h}(1)(\log T)^{1/2+\varepsilon}, 
\]
as stated in Proposition \ref{Bounds-for-J_1-B} and Proposition \ref{bounds-for-J_2-B}.
\subsection{The bounds of $\mathcal{J}_{1}(B)$}

For $\mathcal{J}_1(B)$ , we have $(\log t_{\phi})^{20} \leq |t| \leq t_{\phi}^{\varepsilon}$ then
\[
\widetilde{h}(1 - \frac{1}{\log t_{\phi}}-it) \ll \left(\frac{(\log t_{\phi})^{10}}{1+|s|}\right)^{10} \ll (\log t_{\phi})^{-100}.
\]
By Cauchy-Schwarz inequality, we bound $\mathcal{J}_2(B)$ by
\begin{multline}\label{C-S-inequality}
    (\sum_{j}|\langle \phi^2 , \phi_j\rangle|^2)^{1/2}\times (\log t_{\phi})^{-100}\\
     \left(\frac{1}{2\pi i}\int_{(\log t_{\phi})^{20}\leq |t|\leq t_{\phi}^{\varepsilon}}\sum_{j}\frac{|L(s+\frac{1}{2},\phi_j)L(s +\frac{1}{2}-2iT,\phi_j)|^2}{|\zeta(1 + 2iT)\zeta(1 + 2s - 2iT)|^2}\mathcal{H}(t_j,t,T)\frac{\dd t}{\frac{1}{\log t_{\phi}}+|t|} \right)^{1/2}
\end{multline}
where
\[
\mathcal{H}(t_j,t,T) = \left|\frac{\prod_{\pm}\Gamma(\frac{\frac{1}{2}+\frac{1}{\log t_{\phi}} + it\pm t_j}{2})\Gamma(\frac{\frac{1}{2} + \frac{1}{\log t_{\phi}}+it - 2i T\pm t_j}{2})}{|\Gamma(\frac{1}{2}+it_j)|\Gamma(\frac{1}{2}+iT)\Gamma(\frac{1}{2}+\frac{1}{\log t_{\phi}} + it -iT)}\right|^2.
\]
By Stirling's formula, we get 
\begin{multline}\label{H-function}
    \mathcal{H}(t_j,t,T) \ll \prod_{\pm}(1+|t \pm t_j|)^{1/2+1/\log t_\phi}(1+|t -2T \pm t_j|)^{1/2+1/\log t_\phi}\\ \times(1+ |t - T|)^{1/\log t_{\phi}}\exp(-\frac{\pi}{2}P(t_j,t,T))
\end{multline}
where
\begin{multline*}
P(t_j,t,T)  = \sum_{\pm}(|t \pm t_j|+ |t-2T \pm t_j|) - 2t_j - 2T - 2|t - T|\\
= |t_j+t| +|t_j - t| + |t_{j} - 2T +t| -t_j + t -2T.
\end{multline*}

When $ t\geq 0$. We have
\begin{equation}
        \begin{aligned}
    P(t_j,t,T) =        \left\{\begin{array}{lr}
            2t_j  + 2t - 4T     , \quad &t_{j} \geq 2T-t,  \\
             0,\quad   &  t\leq t_j \leq 2T-t, \\
              t - t_j ,\quad & 0\leq  t_j \leq t.
            \end{array}
            \right.
        \end{aligned}
    \end{equation}

When $ t\leq 0$. We have
\begin{equation}\label{P-function}
        \begin{aligned}
    P(t_j,t,T) =        \left\{\begin{array}{lr}
            2t_j  + 2t - 4T     , \quad &t_{j} \geq 2T-t,  \\
             0,\quad   &  -t\leq t_j \leq 2T-t, \\
              -t - t_j ,\quad & 0\leq  t_j \leq -t.
            \end{array}
            \right.
        \end{aligned}
    \end{equation}
For the contribution from $ (\log t_{\phi})^{20} \leq |t| \leq t_{\phi}^{\varepsilon}$, we have the following proposition.
\begin{proposition}\label{Bounds-for-J_1-B}Under GRH and GRC, we have
    \[
   \mathcal{J}_1(B)\ll \widetilde{h}(1)(\log t_{\phi})^{-10}.
    \]
\end{proposition}

\begin{proof}
When $t \geq 0$, then we truncate the innermost $j$-sum in equation (\ref{C-S-inequality}) to 
\begin{multline}
    \sum_{t - (\log t_{\phi})^2 \leq t_j \leq 2T - t +(\log t_{\phi})^2} \\= \sum_{t - (\log t_{\phi})^2 \leq t_j \leq t  + T^{1-\varepsilon}} + \sum_{t+ T^{1-\varepsilon}\leq t_j \leq t + T+T^{1-\varepsilon}}+\sum_{t+ T+T^{1-\varepsilon}\leq t_j \leq t +2T - T^{2\varepsilon}}+\sum_{ t +2T - T^{2\varepsilon}\leq t_j \leq 2T - t + (\log t_{\phi})^{2}}
\end{multline}
with negligible power saving error term. By the bound of $\mathcal{H}(t_j,t,T)$ and Lemma \ref{P-function}. We consider the contribution of the second line of equation (\ref{C-S-inequality}) from the first two parts called $\sum_1 + \sum_2 + \sum_3 + \sum_4$ where
\begin{equation}
    \begin{aligned}
        {\sum}_1 &: = \sum_{t - (\log t_{\phi})^2 \leq t_j \leq t  + T^{1-\varepsilon}}\frac{|L(\frac{1}{2}+\frac{1}{\log t_{\phi}}+it,\phi_j)L(\frac{1}{2}+\frac{1}{\log t_{\phi}}+it-2iT,\phi_j)|^2}{(1+|t_j|)^{1/2}T(1 + |t_j - t|)^{1/2}},\\
        {\sum}_2 &: =\sum_{t+ T^{1-\varepsilon}\leq t_j \leq t + T+T^{1-\varepsilon}}\frac{|L(\frac{1}{2}+\frac{1}{\log t_{\phi}}+it,\phi_j)L(\frac{1}{2}+\frac{1}{\log t_{\phi}}+it-2iT,\phi_j)|^2}{(1+|t_j|)T},\\
        {\sum}_3 &: =\sum_{t+ T+T^{1-\varepsilon}\leq t_j \leq t +2T - T^{2\varepsilon}}\frac{|L(\frac{1}{2}+\frac{1}{\log t_{\phi}}+it,\phi_j)L(\frac{1}{2}+\frac{1}{\log t_{\phi}}+it-2iT,\phi_j)|^2}{(1+|t_j|)T^{1/2}(1+|t_j - 2T + t|)^{1/2}},\\
        {\sum}_4 &: =\sum_{ t +2T - T^{2\varepsilon}\leq t_j \leq 2T - t + (\log t_{\phi})^{2}}\frac{|L(\frac{1}{2}+\frac{1}{\log t_{\phi}}+it,\phi_j)L(\frac{1}{2}+\frac{1}{\log t_{\phi}}+it-2iT,\phi_j)|^2}{(1+|t_j|)T^{1/2}}.\\
    \end{aligned}
\end{equation}
For $\sum_1$, we get
\begin{multline}
    {\sum}_1 =  \sum_{ - (\log t_{\phi})^2 \leq t_j-t \leq   + T^{1-\varepsilon}}\frac{|L(\frac{1}{2}+\frac{1}{\log t_{\phi}}+it,\phi_j)L(\frac{1}{2}+\frac{1}{\log t_{\phi}}+it-2iT,\phi_j)|^2}{(1+|t_j|)^{1/2}T(1 + |t_j - t|)^{1/2}} \\
    \ll \exp(\frac{C\log T}{\log\log T})\sum_{0\leq k \leq T^{1-2\varepsilon}}\sum_{(k-1)T^{\varepsilon}\leq t_j - t \leq kT^{\varepsilon} }\frac{1}{(1+|t_j|)^{1/2}T(1+(k-1)T^{\varepsilon})^{1/2}}\\
    \ll \exp(\frac{C\log T}{\log\log T}) T^{-1 +\varepsilon}\sum_{10\leq k \leq T^{1-2\varepsilon}}\frac{\sqrt{t + kT^{\varepsilon}}}{((k-1)T^{\varepsilon})^{1/2}}\\
    \ll \exp(\frac{C\log T}{\log\log T}) T^{-1 +\varepsilon}\sum_{10\leq k \leq T^{1-2\varepsilon}}\sqrt{\frac{k+1}{k-1}} \ll T^{-\varepsilon/2}.
\end{multline}
For $\sum_2$, we get
\begin{multline}
{\sum}_2 \leq \sum_{ T^{1-\varepsilon}\leq t_j-t \leq  T + T^{1-\varepsilon}}\frac{|L(\frac{1}{2}+\frac{1}{\log t_{\phi}}+it,\phi_j)L(\frac{1}{2}+\frac{1}{\log t_{\phi}}+it-2iT,\phi_j)|^2}{(1+|t_j|)T}\\
 = \sum_{1\leq k \leq T^{1-\varepsilon}}\sum_{T^{1-\varepsilon}+ (k-1)T^{\varepsilon}\leq t_j - t\leq    T^{1-\varepsilon} +kT^{\varepsilon}}\frac{|L(\frac{1}{2}+\frac{1}{\log t_{\phi}}+it,\phi_j)L(\frac{1}{2}+\frac{1}{\log t_{\phi}}+it-2iT,\phi_j)|^2}{(1+|t_j|)T}\\
 \ll \frac{(\log T)^{2+\varepsilon}}{T}T^{\varepsilon}\sum_{1\leq k \leq T^{1-\varepsilon}-1}1 \ll (\log T)^{2+\varepsilon}.
\end{multline}

Now we deal with $\sum_3$, we get
\begin{multline}
    {\sum}_3 = \sum_{ 2t-T\leq t_j - 2T +t \leq 2t - T^{2\varepsilon}}\frac{|L(\frac{1}{2}+\frac{1}{\log t_{\phi}}+it,\phi_j)L(\frac{1}{2}+\frac{1}{\log t_{\phi}}+it-2iT,\phi_j)|^2}{(1+|t_j|)T^{1/2}(1+|t_j - 2T + t|)^{1/2}} \\
   \ll \sum_{ -T\leq t_j - 2T +t \leq  - T^{\varepsilon}}\frac{|L(s+\frac{1}{2},\phi_j)L(s +\frac{1}{2}-2iT,\phi_j)|^2}{(1+|t_j|)T^{1/2}(1+|t_j - 2T + t|)^{1/2}}\\
    \ll \sum_{1\leq k \leq T^{1-\varepsilon}}\sum_{-(k+1)T^{\varepsilon}\leq t_j -2T +t \leq -kT^{\varepsilon} }\frac{|L(s+\frac{1}{2},\phi_j)L(s +\frac{1}{2}-2iT,\phi_j)|^2}{(1+|t_j|)T^{1/2}(1+(k+1)T^{\varepsilon})^{1/2}}\\
    \ll T^{-1/2 +\varepsilon}(\log T)^{2+\varepsilon}\sum_{1\leq k \leq T^{1-\varepsilon}}\frac{1}{((k+1)T^{\varepsilon})^{1/2}} \ll T^{-1/2 +\varepsilon}(\log T)^{2+\varepsilon} \frac{T^{1/2}}{T^{\varepsilon}} \ll (\log T)^{2+\varepsilon}.
\end{multline}
For $\sum_4$, it is very easy to bound it by $\mathcal{O}(T^{-1/2+\varepsilon})$.

When $t \geq 0$, then very similarly we truncate the sum to 
\begin{multline}
    \sum_{-t - (\log t_{\phi})^2 \leq t_j \leq 2T - t +(\log t_{\phi})^2} \\= \sum_{-t - (\log t_{\phi})^2 \leq t_j \leq -t  + T^{1-\varepsilon}} + \sum_{-t+ T^{1-\varepsilon}\leq t_j \leq -t + T+T^{1-\varepsilon}}\\
    +\sum_{-t+ T+T^{1-\varepsilon}\leq t_j \leq -t +2T - T^{2\varepsilon}}+\sum_{ -t +2T - T^{2\varepsilon}\leq t_j \leq 2T - t + (\log t_{\phi})^{2}}.
\end{multline}
The differ appears only in the third sum but is simpler than the above. The contribution is
\[
     {\sum}_3^{'} \leq \sum_{ -T\leq t_j - 2T +t \leq - T^{\varepsilon}}\frac{|L(s+\frac{1}{2},\phi_j)L(s +\frac{1}{2}-2iT,\phi_j)|^2}{(1+|t_j|)T^{1/2}(1+|t_j - 2T + t|)^{1/2}} \ll (\log T)^{2+\varepsilon}.
\]

In conclusion, we get the innermost $j$-sum in $\mathcal{J}_1(B)$ is bounded by $\mathcal{O}((\log (T + t_{\phi}))^{2+\varepsilon})$. Note that we get 
\[
\sum_{j}|\langle \phi^2 , \phi_j\rangle|^2 \ll 1
\]
under GLH. Then we get
    \[
    \mathcal{J}_1(B) \ll \widetilde{h}(1)(\log t_{\phi})^{-10}.
    \]
Now we complete the proof.
\end{proof}

\subsection{The bounds of $\mathcal{J}_2(B)$}In this subsection, we prove the following proposition.
\begin{proposition}\label{bounds-for-J_2-B}Under GRH and GRC, we get
    \[
    \mathcal{J}_2(B) \ll \widetilde{h}(1)(\log T)^{1/2+\varepsilon}.
    \]
\end{proposition}

By Watson's formula we treat the sum 
\begin{multline}\label{J_2-B}
   \mathcal{J}_2(B)   \ll (\log t_{\phi})^{\varepsilon} \int_{|t| \leq (\log t_{\phi})^{20}} \widetilde{h}(1-\frac{1}{\log t_{\phi}} - it)\sum_{j}L(\frac{1}{2},\phi_j)^{1/2}L(\frac{1}{2},\Sym^2\phi\times\phi_j)^{1/2}\times \\ 
    |L(\frac{1}{2}+\frac{1}{\log t_{\phi}} + it,\phi_j)L(\frac{1}{2}+\frac{1}{\log t_{\phi}}+it-2iT,\phi_j)||\mathcal{H}_1(t_j,t,T)\mathcal{H}_2(t_j,t_{\phi})|\frac{\dd t}{\frac{1}{\log t_{\phi}}+|t|}. 
\end{multline}
Here
\[
|\mathcal{H}_1(t_j, t, T)| = \left|\frac{\prod_{\pm}\Gamma(\frac{\frac{1}{2}+\frac{1}{\log t_{\phi}} + it\pm t_j}{2})\Gamma(\frac{\frac{1}{2} + \frac{1}{\log t_{\phi}}+it - 2i T\pm t_j}{2})}{|\Gamma(\frac{1}{2}+it_j)|\Gamma(\frac{1}{2}+iT)\Gamma(\frac{1}{2}+\frac{1}{\log t_{\phi}} + it -iT)}\right|
\]
and
\[
|\mathcal{H}_2(t_j, t_{\phi})| = \frac{|\Gamma(\frac{\frac{1}{2}+ it_j}{2})|^2\prod_{\pm}|\Gamma(\frac{\frac{1}{2} + 2it_{\phi} \pm it_j}{2})|}{|\Gamma(\frac{1}{2}+it_j)||\Gamma(\frac{1}{2}+it_{\phi})|^2}.
\]
Note that
\[
\widetilde{h}(1-\frac{1}{\log t_{\phi}} - it) = \int_{-\infty}^{\infty}h(A)A^{-\frac{1}{\log t_{\phi}}-it}\dd A \leq \widetilde{h}(1).
\]

By Stirling's formula, we have
    \begin{multline} |\mathcal{H}_1(t_j,t,T)\mathcal{H}_2(t_j,t_{\phi})| \ll  (1+|t_j|)^{-1/2}\prod_{\pm}(1+|t_j \pm 2t_{\phi}|)^{-1/4}\prod_{\pm}(1+|t \pm t_{j}|)^{-1/4+\frac{1}{\log T}}\\\prod_{\pm}(1+|t -2T \pm t_{j}|)^{-1/4+\frac{1}{\log T}}\exp(-\frac{\pi}{2}(Q(t,t_j,T,t_{\phi})))
\end{multline}
where
\begin{multline}
    Q(t,t_j,T,t_\phi) = t_j +\frac{t_j+2t_{\phi}}{2}+\frac{|t_j - 2t_{\phi}|}{2}+\frac{|t_{j}+t|}{2}+\frac{|t_j - t|}{2}\\+\frac{|t-2T+t_j| + |t-2T - t_j|}{2} - 2t_j - 2t_{\phi} - T - |t - T|\\ 
 = \frac{|t_j - 2t_\phi|}{2}+\frac{|t_j+t|}{2} + \frac{|t_j - t|}{2}+\frac{|t_j  -2T+t|}{2} +\frac{t}{2} - t_{\phi} -T.
\end{multline}

Since we restrict $|t| \leq (\log t_{\phi})^{20}$. For three cases $-(\log t_{\phi})^{20}\leq t \leq 0$, $0\leq t \leq 2t_{\phi}-2T$ and $2t_{\phi}-2T\leq t\leq (\log t_{\phi})^{20}$, we have
\begin{lemma}\label{Q-size}Let $  0\leq t \leq (\log t_{\phi})^{20}$, we have
    \begin{equation}\label{Q-t-large}
        \begin{aligned}
    Q(t,t_j,T,t_\phi) =        \left\{\begin{array}{lr}
            2t_j - 2t_{\phi} + t - 2T     , \quad &t_{j} \geq 2t_{\phi},  \\
             t_j- 2T+t,\quad   &  2T - t\leq t_j \leq 2t_{\phi}, \\
              0    ,\quad & t\leq  t_j \leq 2T - t,  \\
             -t_j + t , \quad   &   t_j \leq t.\\
            \end{array}
            \right.
        \end{aligned}
    \end{equation}
Let $2T - 2t_{\phi}\leq t \leq 0$, we have
  \begin{equation}\label{Q-t-middle}
        \begin{aligned}
    Q(t,t_j,T,t_\phi) =        \left\{\begin{array}{lr}
            2t_j - 2t_{\phi} + t - 2T     , \quad &t_{j} \geq 2t_{\phi},  \\
             t_j -2t_\phi,\quad   &  2T-t\leq t_j \leq 2t_{\phi}, \\
             0    ,\quad & -t\leq  t_j \leq 2T -t,  \\
             -t - t_j , \quad   &  t_j \leq -t.\\
            \end{array}
            \right.
        \end{aligned}
    \end{equation}
Let $ -(\log t_{\phi})^{20}\leq t \leq 2T - 2t_{\phi}$, we have
    \begin{equation}\label{Q-t-small}
        \begin{aligned}
    Q(t,t_j,T,t_\phi) =        \left\{\begin{array}{lr}
            2t_j - 2t_{\phi} + t - 2T     , \quad &t_{j} \geq 2T -t\\ t_j - 2T - t,\quad   &  2t_{\phi}\leq t_j \leq 2T - t, \\
             0    ,\quad & -t\leq  t_j \leq 2t_{\phi},  \\
             -t - t_j , \quad   &  t_j \leq -t.\\
            \end{array}
            \right.
        \end{aligned}
    \end{equation}
\end{lemma}
To prove Proposition \ref{bounds-for-J_2-B}, we need the following bound for mixed moments of $L$-functions.
\begin{proposition}\label{Mixed-moment-with-Gamma-factor}Assume GRH and GRC. For $|t| \leq (\log t_{\phi})^{20}$, we have the uniform bound
    \begin{multline}
        \sum_{j}L(\frac{1}{2},\phi_j)^{1/2}L(\frac{1}{2},\Sym^2\phi\times\phi_j)^{1/2} \\ \times
    |L(\frac{1}{2}+\frac{1}{\log t_{\phi}} + it,\phi_j)L(\frac{1}{2}+\frac{1}{\log t_{\phi}}+it-2iT,\phi_j)||\mathcal{H}_1(t_j,t,T)\mathcal{H}_2(t_j,t_{\phi})| \ll (\log (T+t_{\phi}))^{1/2+\varepsilon}.
    \end{multline}
\end{proposition}
\begin{proof}
For the case $ 0 \leq t \leq (\log t_{\phi})^{20}$ and $t$ satisfies $2T - t + (\log t_{\phi})^2 \leq 2t_{\phi}$, we truncate the sum to 
\begin{multline}
    \sum_{t - (\log t_{\phi})^2 \leq t_j \leq 2T - t +(\log t_{\phi})^2} \\= \sum_{t - (\log t_{\phi})^2 \leq t_j \leq t  + T^{1-\varepsilon}} + \sum_{t+ T^{1-\varepsilon}\leq t_j \leq t + T+T^{1-\varepsilon}}+\sum_{t+ T+T^{1-\varepsilon}\leq t_j \leq 2T - t }+\sum_{2T - t\leq t_j \leq 2T - t + (\log t_{\phi})^2}.
\end{multline}
If $t$ satisfies $2T - t + (\log t_{\phi})^2 \geq 2t_{\phi}$, we truncate the sum to
\begin{multline}
    \sum_{t - (\log t_{\phi})^2 \leq t_j \leq 2t_{\phi}+(\log t_{\phi})^2} \\= \sum_{t - (\log t_{\phi})^2 \leq t_j \leq t  + T^{1-\varepsilon}} + \sum_{t+ T^{1-\varepsilon}\leq t_j \leq t + T+T^{1-\varepsilon}}+\sum_{t+ T+T^{1-\varepsilon}\leq t_j \leq 2T - t }+\sum_{2T - t\leq t_j \leq 2t_{\phi}+(\log t_{\phi})^2}.
\end{multline}
The first two parts can treat similarly as the proof of Proposition $\ref{Bounds-for-J_1-B}$. The third sum $\sum\limits_{t+ T+T^{1-\varepsilon}\leq t_j \leq 2T - t }\frac{\mathcal{L}(t_j)}{(1+|t_j|)T^{1/4}t_{\phi}^{1/4}(1 + 2t_{\phi}-t_j)^{1/4}(1 +  2T - t - t_j)^{1/4}}$ contributes
\begin{multline}
    \leq \sum_{ 0 \leq k \leq T^{1-\varepsilon}}\sum_{-(k+1)T^{\varepsilon}\leq t_j  - 2T + t\leq -kT^{\varepsilon}}\frac{\mathcal{L}(t_j)}{|t_j|T^{1/4}t_{\phi}^{1/4}(1+2t_{\phi} - 2T + t +kT^{\varepsilon})^{1/4}(1 + kT^{\varepsilon})^{1/4}}\\
     \ll \frac{T^{\varepsilon}\log (T+t_{\phi})^{1/2+\varepsilon}}{T^{1/4}t_{\phi}^{1/4}}\sum_{10\leq k \leq T^{1-\varepsilon}}\frac{1}{(kT^{\varepsilon})^{1/4}\max\{2t_{\phi} - 2T +1 , kT^{\varepsilon}\}^{1/4}}.
\end{multline}
This term is the same as (\ref{der}) and bounded by
\[
\frac{T^{1/4}\log (T+t_{\phi})^{1/2+\varepsilon}}{t_{\phi}^{1/4}}.
\]
Easily, we get the final sum is bounded by $\mathcal{O}(T^{-1/2+\varepsilon})$ since the sum is very short.



In conclusion, we get the contribution of $0\leq t\leq (\log t_{\phi})^{20}$ is
\[
\frac{T^{1/4}\log (T+t_{\phi})^{1/2+\varepsilon}}{t_{\phi}^{1/4}}.
\]

The case $ (\log t_{\phi})^{20}\leq t \leq 0$ is similar. Then we complete the proof.

\end{proof}

Inserting Proposition \ref{Mixed-moment-with-Gamma-factor} into (\ref{J_2-B}) and we complete the proof of Proposition \ref{bounds-for-J_2-B}. From Proposition \ref{Bounds-for-J_1-B} and Proposition \ref{bounds-for-J_2-B} we complete Proposition \ref{J_A-average}. 

\begin{proof}[The proof of Proposition \ref{equidistribution-average}]The proposition follows by combining (\ref{I}), (\ref{K}) and Proposition \ref{J_A-average}. 
\end{proof}

\section{\label{sec:5}Mixed moments of $L$-functions}
We use Soundararajan's method to prove Theorem \ref{Mixed-moments-Soundararajan's-method-shifted}.

\subsection{Central values of $L$-functions}
Let $\alpha_\phi, \beta_{\phi}$ and $\alpha_j, \beta_j$ be the Satake parameters of $\phi$ and $\phi_j$, respectively. Then we define 
\[
\Lambda_{\phi_j}(p^n) = \alpha_j(p)^n + \beta_j(p)^n,
\]

\[
\Lambda_{\Sym^2\phi\times\phi_j}(p^n) = (\alpha_\phi(p)^{2n} + 1 + \beta_\phi(p)^{2n})(\alpha_j(p)^n + \beta_j(p)^n).
\]
Then we have 
\[
\Lambda_{\phi_j}(p) = \lambda_j(p), \quad \Lambda_{\phi_j}(p^2) = \lambda_j(p)^2 - 2  = \lambda_j(p^2)  - 1 = \lambda_{\Sym^2 \phi_j}(p) - 1,
\]
\[
\Lambda_{\Sym^2 \phi\times \phi_j}(p) = \lambda_{\Sym^2 \phi}(p)\lambda_j(p),
\]
and
\[
\Lambda_{\Sym^2\phi\times\phi_j}(p^2) = (\lambda_{\Sym^4 \phi}(p)- \lambda_{\Sym^2 \phi}(p)+1)(\lambda_{\Sym^2 \phi_j}(p) + 1) .
\]

\begin{lemma}\label{Soundararajan-method-short-sum}Assume GRH. Let $X \leq t_{j} \leq 2X$, $T \leq t_{\phi}$ and $T^{1-\varepsilon}\leq X \leq 3t_{\phi}$. Let $0\leq \Re(z_1) , \Re(z_2) \leq \frac{1}{\log X}$ and $|z_1| , |z_2| \leq 3T$. Then we have
\[
\log L(\frac{1}{2},\phi_j) \leq \sum_{p^n \leq x}\frac{\Lambda_{\phi_j}(p^n)}{np^{n(\frac{1}{2}+\frac{1}{\log x})}}\frac{\log \frac{x}{p^n}}{\log x} + \mathcal{O}(\frac{\log X}{\log x} + 1),
\]
    \[
\log L(\frac{1}{2},\Sym^2 \phi \times\phi_j) \leq \sum_{p^n \leq x}\frac{\Lambda_{\Sym^2\phi\times\phi_j}(p^n)}{np^{n(\frac{1}{2}+\frac{1}{\log x})}}\frac{\log \frac{x}{p^n}}{\log x} + \mathcal{O}(\frac{\log (X + t_\phi)}{\log x} + 1),
    \]
    
\[
\log |L(\frac{1}{2}+z_1,\phi_j)| \leq \Re \sum_{p^n \leq x}\frac{\Lambda_{\phi_j}(p^n)}{np^{n(\frac{1}{2}+\frac{1}{\log x}+z_1)}}\frac{\log \frac{x}{p^n}}{\log x} + \mathcal{O}(\frac{\log (X + T)}{\log x} + 1),
\]
and
\[
\log |L(\frac{1}{2}+z_2,\phi_j)| \leq \Re \sum_{p^n \leq x}\frac{\Lambda_{\phi_j}(p^n)}{np^{n(\frac{1}{2}+\frac{1}{\log x}+z_2)}}\frac{\log \frac{x}{p^n}}{\log x} + \mathcal{O}(\frac{\log (X + T)}{\log x} + 1).
\]
\end{lemma}
\begin{proof}See \cite{MR2538566} \cite{MR2552116}.
\end{proof}
For a real number $x \geq 10$ and complex number $z$, we define
\begin{equation}
    \begin{aligned}
      \mathcal{N}(z , x) =   \left\{\begin{array}{lr}
             \log\log x, \quad & |z| \leq \frac{1}{\log x},  \\
           -\log |z|, \quad  & \frac{1}{\log x}\leq |z| \leq 1,\\
            0, \quad & |z |\geq 1.
        \end{array}
        \right.
    \end{aligned}
\end{equation}
Moreover, we define
\begin{equation}\label{Mean-value-L-function}
    \mathcal{M}(z_1,z_2, x) = \mathcal{N}(z_1 , x) + \mathcal{N}(z_2 , x).
\end{equation}

\begin{lemma}\label{lemma:sumofcoefficients}
  Assume GRH and GRC. 
  Let $\phi , \phi_j$ be two distinct Hecke--Maass cusp forms over $\SL_2(\mathbb{Z})$. Let $X \leq t_{j} \leq 2X$, $1\leq t_\phi^{\delta}\leq T \leq t_{\phi}$ and $T^{1-\varepsilon}\leq X \leq 3t_{\phi}$. Suppose that  $0\leq \Re(z_1)\leq \frac{1}{\log X}$ and $|z_1|\leq 3T$.  For $x\geq 2$ and $x \leq X^2$, we have
\begin{equation}\label{eqn:sumofcoefficients1a}
  \sum_{p\leq x}\frac{\lambda_{\sym^4\phi}(p)\lambda_{\sym^2 \phi_j}(p)}{p}=O(\log\log\log (t_j+t_\phi)),
  \end{equation}
\begin{equation}\label{eqn:sumofcoefficients2a}
    \sum_{p\leq x}\frac{\lambda_{\sym^2\phi}(p)
    \lambda_{\sym^2 \phi_j}(p)}{p}=O(\log\log\log (t_j+t_\phi)),
\end{equation}

\begin{equation}\label{eqn:sumofcoefficients4a}
    \sum_{p\leq x}\frac{\lambda_{\sym^2 \phi}(p)}{p}=O(\log\log\log t_\phi),
\end{equation}
\begin{equation}\label{eqn:sumofcoefficients4b}
    \sum_{p\leq x}\frac{\lambda_{\sym^2 \phi_j}(p)}{p}=O(\log\log\log t_j).
\end{equation}
\begin{equation}\label{eqn:sumofcoefficients4c}
    \sum_{p\leq x}\frac{\lambda_{\sym^2 \phi_j}(p)}{p^{1 +z_1}}=O(\log\log\log (t_j +T)).
\end{equation}

\begin{equation}\label{eqn:sumofcoefficients4d}
    \sum_{p\leq x}\frac{1}{p^{1 +z_1}}  = \mathcal{N}(z_1,x) + \mathcal{O}(\log\log\log T).
\end{equation}

\end{lemma}
\begin{proof}We only prove (\ref{eqn:sumofcoefficients4c}) and (\ref{eqn:sumofcoefficients4d}). The proofs of (\ref{eqn:sumofcoefficients1a}) (\ref{eqn:sumofcoefficients2a})(\ref{eqn:sumofcoefficients4a}) (\ref{eqn:sumofcoefficients4b}) appear in \cite{hua2024jointvaluedistributionheckemaass}.

    Note that $L(s,\Sym^2\phi_j)$ has an analytic continuation to the complex plane and satisfies a functional equation.
  Assume GRH for $L(s,\Sym^2\phi_j)$. It follows that $\log L(s,\Sym^2\phi_j)$ is analytic in the region $\Re (s)\geq \frac{1}{2}+\frac{1}{\log x}$. Moreover, by repeating a classical argument of Littlewood (see Titchmarsh \cite[(14.2.2)]{MR882550}), in this region we have
  \begin{equation}\label{eqn:log L <<}
  |\log L(s,\Sym^2u_j)|
  \ll \Big(\Re(s)-\frac{1}{2}\Big)^{-1}\log (t_j+|\Im(s)|).
  \end{equation}
  By Perron's formula, we have for $x\geq 2$ that
 \begin{equation}
\begin{aligned}
             \sum_{p\leq x}\frac{\lambda_{\sym^2 \phi_j}(p)}{p^{1+ z_1}}&
    =\frac{1}{2\pi i}
    \int_{1-ix\log(t_j+ T +x)}
    ^{1+ix\log(t_j+ T+x)}
    \log L(s+1 + z_1,\Sym^2\phi_j)
    x^s\frac{\dd s}{s}+\mathcal{O}(1).
\end{aligned}
 \end{equation}
  Shifting contours to $\Re(s)=-\frac{1}{2}+\frac{1}{\log x}$ and we collect a simple pole at $s=  0$ with residue $\log L(1+z_1,\Sym^2\phi_j)$. The upper horizonal contour is bounded by
  \begin{multline}
    \ll\frac{1}{| x( \log(t_j + T +x))|}
    \int_{-\frac{1}{2}+\frac{1}{\log x}
    +ix\log(t_j+T+x)}^{1+ix\log(t_j+T+x)}
    |\log L(s+1+z_1,\Sym^2\phi_j)|
    |x^s| |\dd s|\\
    \ll \frac{\log x\log(t_j + T +x\log(t_j+T+x))}
    {| x( \log(t_j + T +x))|}
    \int_{-\frac{1}{2}}^{1}x^u \dd u\ll 1,
  \end{multline}
and the lower horizontal contour is also $O(1)$.
Hence by \eqref{eqn:log L <<} we have for $x\geq 2$ that
\begin{multline}
   \sum_{p\leq x}\frac{    \lambda_{\Sym^2 \phi_j}(p)}{p^{1+z_1}}\\
    =\log L(1 +z_1,\Sym^2\phi_j)
    +\mathcal{O}\Big(\frac{\log x}{\sqrt{x}}
    \int_{ -x\log(t_j+T+x)}^{ x\log(t_j+T+x)}
    \frac{\log (t_j + T + |u|)}{1+u}\dd u \Big)\\
    = \log L(1 +z_1,\Sym^2\phi_j)
    +\mathcal{O}\Big(\frac{(\log x)^2 \log (t_j +T)}{\sqrt{x}} \Big)
\end{multline}
  Applying the above estimate twice we have for $y\geq (\log(T+t_j))^{10}$ that
  \begin{equation}
    |\sum_{(\log (X+t_j))^{10}<p\leq y}
    \frac{
    \lambda_{\Sym^2 \phi_j}(p)}{p^{1+z_1}}|
    \ll 1.
  \end{equation}
  Assuming GRC, for $y\leq (\log (T+t_j))^{10}$ we have
  \begin{equation}
   |\sum_{p\leq y}
    \frac{
    \lambda_{\Sym^2 \phi_j}(p)}{p^{1+z}}|
    \ll
    \log\log\log (T+t_j).
  \end{equation}
  Hence we get (\ref{eqn:sumofcoefficients4c}).
  
Now we prove (\ref{eqn:sumofcoefficients4d}). By Prime Number Theorem, we get
\[
\sum_{p \leq x}\frac{1}{p^{1+z_1}} = \int_{2}^{x}\frac{1}{t^{1+z_1}\log t}\dd t + \mathcal{O}(1).
\]
When $ |z_1| \leq (\log x)^{-1}$, we get
\begin{multline}
    \int_{2}^{x}\frac{1}{t^{1+z_1}\log t}\dd t = \int_{2}^{x}\frac{1}{t\log t}\sum_{n\geq 0}\frac{(-z_1\log t)^n}{n!} \dd t\\ = \log\log x + \sum_{n\geq 1}\frac{(-z_1\log x)^{n}}{n!n} = \log\log x + \mathcal{O}(1).
\end{multline}
When $(\log x)^{-1}\leq|z_1|\leq 1$, we note that 
\[
\sum_{ n \geq 1}\frac{(-y)^n}{n! n} = \int_{1}^{y}\frac{e^{-s} - 1}{s}\dd s +\mathcal{O}(1) = -\log y + \int_{1}^{y}\frac{e^{-s}}{s}\dd s +\mathcal{O}(1) .
\]
Then by Cauchy integration theorem we have
  \[
  \sum_{n \geq 1}\frac{(-z_1\log x)^{n}}{n! n} = -\log (|z_1|\log x) +\mathcal{O}(\int_{1}^{+\infty}+\int_{i\Im(z_1)\log x+\Re(z_1)\log x
 }^{i\Im(z_1)\log x+\infty}\frac{e^{-\sigma}}{|s|}\dd s) +\mathcal{O}(1).
  \]
  Since $1\leq |z_1|\log x \leq \log x$, we get $|s| \geq 1$. Then we bound the two integral by absolute constant.
Now we get
\[
\sum_{p \leq x}\frac{1}{p^{1+z_1}} = -\log|z_1| + \mathcal{O}(1).
\]

 When $ |z_1| \geq 1$, we must have $|\Re(z_1)| \geq \frac{\sqrt{2}}{2}$ or $|\Im(z_1)| \geq \frac{\sqrt{2}}{2}$. The case of $|\Re(z_1)| \geq \frac{\sqrt{2}}{2}$ is quite easy. Assume $|\Im(z_1)| \geq \frac{\sqrt{2}}{2}$. We follow the very similar process in the proof of (\ref{eqn:sumofcoefficients4c}). And we get in this case
 \[
 \sum_{p \leq x}\frac{1}{p^{1+z_1}} \ll \log\log\log T.
 \]
 In conclusion, we have
 \[
  \sum_{p\leq x}\frac{1}{p^{1 +z_1}}  = \mathcal{N}(z_1,x) + \mathcal{O}(\log\log\log T).
 \]
 This completes the proof.
\end{proof}

From Lemma \ref{Soundararajan-method-short-sum} and Lemma \ref{lemma:sumofcoefficients} we can restrict the short Dirichlet polynomials to the sum over prime.

\begin{lemma}\label{Central-values}Assume GRH and GRC. Let $X \leq t_{j} \leq 2X$, $T \leq t_{\phi}$ and $t_{\phi}^{1-\varepsilon}\leq X \leq 3t_{\phi}$. Let $0\leq \Re(z_1) , \Re(z_2) \leq \frac{1}{\log X}$ and $|z_1| , |z_2| \leq 3T$, then we have
\[
\log L(\frac{1}{2},\phi_j) \leq \sum_{p \leq x}\frac{\Lambda_{\phi_j}(p)}{p^{\frac{1}{2}+\frac{1}{\log x}}}\frac{\log \frac{x}{p}}{\log x} - \frac{1}{2}\log\log x +\mathcal{O}(\log\log\log X) + \mathcal{O}(\frac{\log X}{\log x} + 1),
\]
\begin{multline*}
    \log L(\frac{1}{2},\Sym^2 \phi \times\phi_j) \leq \sum_{p \leq x}\frac{\Lambda_{\Sym^2\phi\times\phi_j}(p)}{p^{\frac{1}{2}+\frac{1}{\log x}}}\frac{\log \frac{x}{p}}{\log x} -\frac{1}{2}\log\log x \\+ \mathcal{O}(\log\log\log(X + t_{\phi}))+ \mathcal{O}(\frac{\log (X + t_\phi)}{\log x} + 1),
\end{multline*}  
and
\begin{multline}
    \log |L(\frac{1}{2}+z_1,\phi_j)L(\frac{1}{2}+z_2,\phi_j)| \leq \Re \sum_{p \leq x}\frac{\Lambda_{\phi_j}(p)(p^{-z_1}+p^{-z_2})}{p^{\frac{1}{2}+\frac{1}{\log x}}}\frac{\log \frac{x}{p}}{\log x} -\frac{1}{2}\mathcal{M}(z_1,z_2 , x)\\ +\mathcal{O}(\log\log\log (X + T)) + \mathcal{O}(\frac{\log (X + T)}{\log x} + 1).
\end{multline}
\end{lemma}

\begin{proof}
    
Under GRC, we can restrict the sum to $p^n$, $n \leq 2$ that is
\[
\sum_{\substack{p^n \leq x\\n\geq 3}}\frac{\Lambda_{\phi_j}(p^n)}{np^{n(\frac{1}{2}+\frac{1}{\log x})}}\frac{\log \frac{x}{p^n}}{\log x} \ll 1.
\]
We also assume $x \leq X^2$, and get
\[
\log L(\frac{1}{2},\phi_j) \leq \sum_{p \leq x}\frac{\Lambda_{\phi_j}(p)}{p^{(\frac{1}{2}+\frac{1}{\log x})}}\frac{\log \frac{x}{p}}{\log x} + \sum_{p \leq \sqrt{x}}\frac{\Lambda_{\phi_j}(p^2)}{2p^{2(\frac{1}{2}+\frac{1}{\log x})}}\frac{\log \frac{x}{p^2}}{\log x} +  \mathcal{O}(\frac{\log X}{\log x} + 1).
\]
The second sum is
\[
 \sum_{p \leq \sqrt{x}}\frac{\lambda_{\Sym^2 \phi_j}(p) - 1}{2p^{1+\frac{2}{\log x}}}\frac{\log \frac{x}{p^2}}{\log x} =  \frac{1}{2}\sum_{p \leq \sqrt{x}}\frac{\lambda_{\Sym^2 \phi_j}(p) }{p^{1+\frac{2}{\log x}}}\frac{\log \frac{x}{p^2}}{\log x} - \frac{1}{2}\sum_{p \leq \sqrt{x}}\frac{1}{p^{1+\frac{2}{\log x}}}\frac{\log \frac{x}{p^2}}{\log x}.
\]
We have
\begin{equation}
    \begin{aligned}
        \sum_{p \leq \sqrt{x}}\frac{1}{p^{1+\frac{2}{\log x}}}\frac{\log \frac{x}{p^2}}{\log x} &= \int_{2}^{\sqrt{x}}\frac{1}{t^{1+\frac{2}{\log x}}}(1 - \frac{\log t^2}{\log x})\dd \pi(t)\\
        & = \int_{2}^{\sqrt{x}}\frac{1}{t^{1+\frac{2}{\log x}}\log t} \dd t + \mathcal{O}(1)\\
        & = \sum_{n \geq 0}\frac{1}{n!}\int_{2}^{\sqrt{x}}\frac{(-\frac{2\log t}{\log x})^n}{t\log t} \dd t \\
        & = \log\log x + \sum_{n \geq 1}\frac{(-2)^n}{n!}\int_{2}^{\sqrt{x}}\frac{(\frac{\log t}{\log x})^n}{t\log t} \dd t \\
        & = \log\log x  + \sum_{n \geq 1}\frac{(-2)^n}{n!}\frac{(\frac{1}{2}\log x)^n}{(\log x)^n} = \log\log x +O(1).
    \end{aligned}
\end{equation}
And 
\begin{multline*}
    \sum_{p \leq \sqrt{x}}\frac{\lambda_{\Sym^2 \phi_j}(p) }{p^{1+\frac{2}{\log x}}}\frac{\log \frac{x}{p^2}}{\log x} = \int_{2}^{\sqrt{x}}\sum_{p\leq t}\frac{\lambda_{\Sym^2 \phi_j}(p)}{p}t^{-\frac{2}{\log x}}(\frac{4}{t \log x} - \frac{4\log t}{t (\log x)^2})\dd t + \mathcal{O}(1) \\ 
    \ll \log\log\log t_{\phi}.
\end{multline*}
Then we get
\[
\log L(\frac{1}{2},\phi_j) \leq \sum_{p \leq x}\frac{\Lambda_{\phi_j}(p)}{p^{(\frac{1}{2}+\frac{1}{\log x})}}\frac{\log \frac{x}{p}}{\log x} - \frac{1}{2}\log\log x +\mathcal{O}(\log\log\log t_{\phi}).
\]
Similarly, from
\[
 \sum_{p \leq \sqrt{x}}\frac{\Lambda_{\Sym^2 \phi \times \phi_j}(p^2)}{2p^{2(\frac{1}{2}+\frac{1}{\log x})}}\frac{\log \frac{x}{p^2}}{\log x} = \sum_{p \leq \sqrt{x}}\frac{(\lambda_{\Sym^4 \phi}(p)- \lambda_{\Sym^2 \phi}(p)+1)(\lambda_{\Sym^2 \phi_j}(p) + 1)}{2p^{2(\frac{1}{2}+\frac{1}{\log x})}}\frac{\log \frac{x}{p^2}}{\log x}
\]
and Lemma \ref{lemma:sumofcoefficients}, we get
\[
\log L(\frac{1}{2},\Sym^2 \phi \times\phi_j) \leq \sum_{p \leq x}\frac{\Lambda_{\Sym^2\phi\times\phi_j}(p)}{p^{(\frac{1}{2}+\frac{1}{\log x})}}\frac{\log \frac{x}{p}}{\log x} -\frac{1}{2}\log\log x + \mathcal{O}(\log\log\log(X + t_\phi )).
\]

Moreover, we have
\begin{multline}
    \log|L(\frac{1}{2}+z_1,\phi_j)L(\frac{1}{2}+z_2,\phi_j)| \leq \\\Re\sum_{p \leq x}\frac{\Lambda_{\phi_j}(p)(p^{-z_1 } + p^{-z_2})}{p^{\frac{1}{2}+\frac{1}{\log x}}}\frac{\log \frac{x}{p}}{\log x} -\frac{1}{2} \mathcal{M}(z_1,z_2,x) +\mathcal{O}(\log\log\log (t_{\phi} + T))
\end{multline}
from 
\begin{multline}\label{Prime-Sym}
    \sum_{p \leq \sqrt{x}}\frac{\lambda_{\Sym^2 \phi_j}(p)(p^{-2z_1 }+p^{-2z_2 }) }{p^{1+\frac{2}{\log x}}}\frac{\log \frac{x}{p^2}}{\log x}\\ = \int_{2}^{\sqrt{x}}\sum_{p\leq t}\frac{\lambda_{\Sym^2 \phi_j}(p)(p^{-2z_1 }+p^{-2z_2 })}{p}t^{-\frac{2}{\log x}}(\frac{4}{t \log x} - \frac{4\log t}{t (\log x)^2})\dd t + \mathcal{O}(1)\\ \ll \log\log\log (X+T)   
\end{multline}
and
\begin{multline}\label{Prime-mathcal-M}
        \sum_{p \leq \sqrt{x}}\frac{p^{-2z_1 }+p^{-2z_2 }}{p^{1+\frac{2}{\log x}}}\frac{\log \frac{x}{p^2}}{\log x} = \int_{2}^{\sqrt{x}}\sum_{p\leq t}\frac{p^{-2z_1 }+p^{-2z_2 }}{p}t^{-\frac{2}{\log x}}(\frac{4}{t \log x} - \frac{4\log t}{t (\log x)^2})\dd t + \mathcal{O}(1)\\
        = \mathcal{M}(z_1,z_2,x) +  \mathcal{O}(\log\log \log (X+T) )
\end{multline}
where 
\[
\mathcal{M}(z_1,z_2,x) = \mathcal{N}(z_1,x) + \mathcal{N}(z_2, x).
\]

We now discuss (\ref{Prime-mathcal-M}). We only need consider the sum containing $z_1$. When $|z_1| \geq 1 $, the upper bound is directly from (\ref{eqn:sumofcoefficients4d}). When $ \frac{1}{\log x}\leq|z_1| \leq 1$, we return the partial summation process and get
\begin{multline}
    -\int_{2}^{e^{\frac{1}{2|z_1|}}}(\log\log t) \dd (\frac{\log \frac{x}{t^2}}{\log x}t^{-\frac{2}{\log x}})-\int_{e^{\frac{1}{2|z_1|}}}^{\sqrt{x}}(-\log|2z_1|) \dd (\frac{\log \frac{x}{t^2}}{\log x}t^{-\frac{2}{\log x}})\\
     = \int_{2}^{e^{\frac{1}{2|z_1|}}}\frac{1}{t\log t}(1 - \frac{2\log t}{\log x})t^{-\frac{2}{\log x}}\dd t = \int_{2}^{e^{\frac{1}{2|z_1|}}}\frac{1}{t\log t}(1 - \frac{2\log t}{\log x})\sum_{ n \geq 0 }\frac{(\frac{-2\log t}{\log x})^n}{n!}\dd t\\
      = -\log |z_1| + \mathcal{O}(\log\log\log (X+T)).
\end{multline}
When $|z_1| \leq \frac{1}{\log x}$, as before we get
\[
    -\int_{2}^{\sqrt{x}}(\log\log t) \dd (\frac{\log \frac{x}{t^2}}{\log x}t^{-\frac{2}{\log x}}) = \log\log x + \mathcal{O}(\log\log\log (X+T)).
\]
Now we complete our proof.

\end{proof}

\subsection{\label{Density-function-subsection}Reduction to density function}

We define 
\[
\mathcal{L}(t_j) =  L(\frac{1}{2},\phi_j)^{\ell_1}L(\frac{1}{2},\Sym^2 \phi \times\phi_j)^{\ell_2}|L(\frac{1}{2}+z_1,\phi_j)L(\frac{1}{2}+z_2,\phi_j)|^{\ell_3}
\]
and the density function
\[
\mathcal{B}_{X, Y}(V) = \#\{X\leq t_j \leq X+Y :\log \mathcal{L}(t_j) > V, \phi_j \neq \phi\}.
\]
We have the identity
\[
    \sum_{X\leq t_j \leq X+Y}\mathcal{L}(t_j) = \int_{\mathbb{R}}e^{V}\mathcal{B}_{X,Y}(V) \dd V + \delta_{X \leq t_\phi \leq X+Y}D_{\phi}\mathcal{L}(t_{\phi})
\]
where $D_{\phi}$ is the number of eigenfunctions with spectral parameter $t_{\phi}$. By the method in \cite{MR2538566} and \cite{MR2781205}, we have
\[
L(\frac{1}{2}, \pi) \ll \exp(\frac{c\log C(\pi)}{\log\log C(\pi)}).
\]
where $C(\pi)$ is the analytic conductor of $\pi$. So by Weyl law we get
\[
\delta_{X \leq t_\phi \leq X+Y}D_{\phi}\mathcal{L}(t_{\phi}) \ll Y e^{\frac{C\log (X+t_\phi)}{\log\log(X+t_{\phi})}}.
\]

Now we only consider
\[
\sqrt{\log\log(X+t_\phi)} \leq V \leq C\frac{\log(X + t_{\phi})}{\log\log(X + t_{\phi })}
\]
and we need estimate the size of $B_{X,Y}(V)$ in this range. We define
\[
\mu(X) = (-\frac{1}{2}+\varepsilon)(\ell_1+\ell_2)\log\log (X +t_{\phi}) - \frac{\ell_3}{2} \mathcal{M}(z_1,z_2,X+t_{\phi}).
\]
From Lemma \ref{Central-values} we have
\begin{multline}
    \log\mathcal{L}(t_j) =  \ell_1\log L(\frac{1}{2},\phi_j) + \ell_2 \log L(\frac{1}{2},\Sym^2 \phi \times\phi_j)+ \ell_3 \log|L(\frac{1}{2}+z_1,\phi_j)L(\frac{1}{2}+z_2,\phi_j)|\\
    \leq \sum_{p \leq x}\frac{(\ell_1 +  \ell_2 \lambda_{\Sym^2\phi}(p) )\lambda_j(p)  }{p^{(\frac{1}{2}+\frac{1}{\log x})}}\frac{\log \frac{x}{p}}{\log x}  + \Re\sum_{p \leq x}\frac{\ell_3 (p^{-z_1} + p^{-z_2})\lambda_j(p)  }{p^{(\frac{1}{2}+\frac{1}{\log x})}}\frac{\log \frac{x}{p}}{\log x}\\ -\frac{1}{2}(\ell_1 + \ell_2)\log\log x -\frac{1}{2}\ell_3 \mathcal{M}(z_1,z_2 ,x)
    + \mathcal{O}(\log\log\log(X +t_\phi)).
\end{multline}
We define
\[
\mathcal{P}(t_j ; x , y) = \sum_{p \leq y}\frac{(\ell_1 +  \ell_2 \lambda_{\Sym^2\phi}(p) + \ell_3 (\frac{p^{-z_1} + p^{-\overline{z_1}  } + p^{-z_2} + p^{-\overline{z_2}  }}{2}))\lambda_j(p)  }{p^{(\frac{1}{2}+\frac{1}{\log x})}}\frac{\log \frac{x}{p}}{\log x}
\]
and
\[
\mathcal{A}_{X,Y}(V,x) =\#\{X< t_j\leq X+Y : \mathcal{P}(t_j ; x, x) > V \}.
\]

Now let $x = X^{\frac{1}{\varepsilon V}}$. When $\sqrt{\log\log X} \leq V \leq (\log\log(X+t_{\phi}))^4$, we get
\[
-\frac{1}{2}(\ell_1+\ell_2)\log\log(X+t_\phi) -\frac{\ell_3}{2} \mathcal{M}(z_1,z_2,X+t_{\phi})+ \log\log\log(X + t_{\phi}) \ll \mu(X).
\]
Then
\begin{equation}\label{BXY<AXY}
    \mathcal{B}_{X,Y}(V+\mu(X)) \leq \mathcal{A}_{X,Y}(V(1-2\varepsilon),x).
\end{equation}
For $ V \geq (\log\log(X+t_{\phi}))^4$, we use $V+ \mu(X) = V(1+o(1))$ and then the inequality above is also true.

\subsection{A Mean value bound of Dirichlet polynomials}Before estimating the density function $\mathcal{A}_{X,Y}(V(1-2\varepsilon), x)$, we need a mean value theorem of the moments of Dirichlet polynomials. 

We introduce the harmonic weights
\[
  \omega_j = \frac{4\pi |\rho_j(1)|^2 }{ \cosh(\pi t_j) } = \frac{2\pi}{L(1,\Sym^2 \phi_j)}
  \quad  \textrm{and} \quad
  \omega(t) = \frac{4\pi |\rho_t(1)|^2}{\cosh(\pi t)} = \frac{4\pi}{|\zeta(1+2it)|^2}.
\]
 For any test function
$h(t)$ which is even and satisfies the following conditions:
\begin{itemize}
  \item [(i)] $h(t)$ is holomorphic in $|\Im(t)|\leq 1/2+\varepsilon$,
  \item [(ii)] $h(t)\ll (1+|t|)^{-2-\varepsilon}$ in the above strip,
\end{itemize}
we have the following Kuznetsov formula (see \cite[Eq. (3.17)]{MR1779567} for example).
\begin{lemma}\label{lemma: KTF}
  For $m,n\geq1$, we have
  \begin{equation*}
    \begin{split}
        & {\sum_j}h(t_j)\omega_j \lambda_j(m)\lambda_j(n) + \frac{1}{4\pi}\int_{-\infty}^{\infty}h(t)\omega(t)\eta_t(m)\eta_t(n)\dd t \\
        & \hskip 120pt = \frac{\delta_{m,n}H^{0}}{\pi} +\sum_{c\geq1}\frac{1}{c} S(n, m;c)H^{+}\left(\frac{4\pi\sqrt{mn}}{c}\right),
    \end{split}
  \end{equation*}
 where $\delta_{m,n}$ is the Kronecker symbol,
 \begin{equation}\label{eqn: H}
    \begin{split}
       H^{0} & = \int_{-\infty}^{+\infty} h(t) \tanh(\pi t)t \dd t, \\
       H^+(x) & = 2i \int_{-\infty}^{\infty} J_{2it}(x)\frac{h(t)t}{\cosh(\pi t)} \dd t, 
    \end{split}
  \end{equation}
  and $J_\nu(x)$ is the standard $J$-Bessel function.
\end{lemma}

Let $X$ be a large real number with $t_{\phi}^{\delta} \leq X \leq 3t_{\phi}$ and $ X^{\varepsilon}\leq Y \leq X$. 
We define the weight functions
    \[
    h(t ; X , Y):=e^{-\left(\frac{t-X}{Y}\right)^2}
    +e^{-\left(\frac{t+X}{Y}\right)^2},
    \]
then we have the following auxiliary lemma.
\begin{lemma}\label{lemma:3.1}
    For any positive integer $n\ll X^{2-\varepsilon}$, we have
    \begin{equation}
    {\sum_{j\geq 1}}w_j\lambda_{j}(n)h(t_j;X,Y) = \delta_{n,1}(\frac{2}{\sqrt{\pi}}XY + \mathcal{O}(Y^2))+O(X^{\varepsilon}Y),
    \end{equation}
    where $w_j=\frac{2\pi}{L(1,\sym^2 \phi_j)}$.
\end{lemma}
\begin{proof}By Kuznestov trace formula Lemma \ref{lemma: KTF}, we get the left side is
\begin{multline}
    \frac{\delta_{n,1}}{\pi}\int_{-\infty}^{+\infty}h(t;X,Y)(\tanh\pi t)t \dd t + \sum_{c\geq1}\frac{1}{c} S(n, 1;c)H^{+}\left(\frac{4\pi\sqrt{n}}{c}\right)\\ - \frac{1}{4\pi}\int_{-\infty}^{\infty}h(t;X,Y)\omega(t)\eta_t(n)\dd t.
\end{multline}
    The Eisenstein part is bounded by
    \[
    X^{\varepsilon}\int_{-\infty}^{\infty}h(t;X,Y) \dd t \ll X^{\varepsilon}Y
    \]
    by definition of $h(t;X,Y)$. Recall 
    \[
     H^+(x)  = 2i \int_{-\infty}^{\infty} J_{2it}(x)\frac{h(t)t}{\cosh(\pi t)} \dd t.
    \]
  By \cite[Lemma 7.1]{MR3635360}, we have bound $H^{+}(\frac{4\pi\sqrt{n}}{c}) \ll Y\frac{\sqrt{n}}{c}$ if $\frac{\sqrt{n}}{c} \leq X$. We also have $H^{+}(\frac{4\pi\sqrt{n}}{c}) \ll X^{-10}$  when $\frac{\sqrt{n}}{c} \ll Y X^{1-\varepsilon}$. Note that $\sqrt{n}\ll X^{1-\varepsilon}$. We use Weil bound for Kloosterman sum. Then we get the contribution of $J$-Bessel function is
 \[
 \sum_{c \leq X^{10}}\frac{\tau(c)}{\sqrt{c}}X^{-10}+\sum_{c \geq X^{10}}\frac{\tau(c)}{\sqrt{c}}Y\frac{\sqrt{n}}{c}\ll X^{-1}.
 \]
 From direct calculation, we get 
 \[
 \int_{-\infty}^{+\infty}h(t;X,Y)(\tanh\pi t)t \dd t = 2\sqrt{\pi} XY + \mathcal{O}(Y^2).
 \]
\end{proof}
\begin{lemma}\label{lemma:sumusesummationformula}
  Let $r\in\mathbb{N}$. Then for $x\leq X^{\frac{1}{10r}}$ and real numbers $a_p\ll p^{\frac{1}{2}}$, we have that
\[
    \sum_{j} \frac{h(t_j;X,Y)}{L(1,\sym^2\phi_j)}
    \Big(\sum_{p\leq x}
    \frac{a_p\lambda_{j}(p)}
    {p^{\frac{1}{2}}}\Big)^{2r}
    \ll G(X,Y)\frac{(2r)!}{r!2^r}
    \Big(\sum_{p\leq x}\frac{a_p^2}{p}\Big)^r+X^{1/2+\varepsilon}Y.
\]
where $G(X,Y) = \frac{2}{\sqrt{\pi}}XY + \mathcal{O}(Y^2)$.
\end{lemma}
\begin{proof}At first we have
\[
 \Big(\sum_{p\leq x}
    \frac{a_p\lambda_{j}(p)}
    {p^{\frac{1}{2}}}\Big)^{2r} = \sum_{n \leq x^{2r}}\frac{a_{2r,x}(n)}{n^{1/2}}\prod\limits_{n =\prod_i p_i^{\alpha_i}} \lambda_{j}(p_{i})^{\alpha_i}
\]
where 
\begin{equation}
   a_{2r,x}(n)=
    \begin{cases}
     \binom{2r}{\alpha_1,\ldots,\alpha_s} \prod_{i=1}^{s} a(p_i)^{\alpha_i},    & \text{if } n  = \prod_{i}p_i^{\alpha_i}, p_i \leq x \text{ and } \sum_{i}\alpha_i = 2r,
    \\
    0  ,&\text{otherwise.}
    \end{cases}
\end{equation}

  By the Hecke relations and some computations, we know (see e.g. \cite[Lemma 7.1]{MR2825478})
  \[
    \lambda_j(p)^{\alpha} = \sum_{l=0}^{\alpha/2} \frac{\alpha! (2l+1)}{(\alpha/2-l)! (\alpha/2+l+1)!} \lambda_j(p^{2l}),
  \]
  if $2\mid \alpha$; and
  \[
    \lambda_j(p)^{\alpha} =\sum_{l=0}^{(\alpha-1)/2} \frac{\alpha! (2l+2)}{((\alpha-1)/2-l)! ((\alpha+3)/2+l)!} \lambda_j(p^{2l+1}),
  \]
  if $2\nmid \alpha$.
  So we can write  $\lambda_j(p)^{\alpha}=\sum_{\beta=0}^{\alpha} b_{\alpha,\beta} \lambda_{j}(p^{\beta})$, where
  \[
    b_{\alpha,\beta} = \left\{
    \begin{array}{ll}
      \frac{\alpha! (\beta+1)}{(\frac{\alpha-\beta}{2})! (\frac{\alpha+\beta}{2}+1)!}, & \textrm{if } \alpha \equiv \beta \ \mod 2,\ 0\leq \beta\leq \alpha, \\
      0, & \textrm{otherwise.}
    \end{array}
    \right.
  \]
  Hence we obtain
  \[
    \left( \sum_{p\leq x} \frac{a_p \lambda_j(p)}{p^{1/2}} \right)^{2r}
    = \sum_{n\leq x^{2r}} \frac{a_{2r,x}(n)}{n^{1/2}} \sum_{m\mid n} b(m,n) \lambda_j(m),
  \]
  where $b(m,n)= \prod_{i=1}^{s} b_{\alpha,\beta}$ for $m=\prod_{i=1}^{s}p_i^{\beta_i}$.
  In particular, $b(1,n)=0$ unless $2\mid \alpha_i$ for all $i$,
  in which case we have $b(1,n)=\prod_{i=1}^{s} \frac{\alpha_i !}{(\alpha_i/2)! (\alpha_i/2+1)!}$.
  Note that we have
  \begin{equation}\label{eqn:b}
    0\leq b_{\alpha,\beta}\leq \binom{\alpha}{(\alpha-\beta)/2}\leq  2^{\alpha}, \quad
    0\leq b(m,n)\leq \prod_{i=1}^s 2^{\alpha_i}\leq n.
  \end{equation}
Now we define
\begin{multline*}
      \sum_{j} \frac{h(t_j;X,Y)}{L(1,\sym^2\phi_j)}
    \Big(\sum_{p\leq x}
    \frac{a_p\lambda_{j}(p)}
    {p^{\frac{1}{2}}}\Big)^{2r} 
    =  \sum_{j} \frac{h(t_j;X,Y)}{L(1,\sym^2\phi_j)} \sum_{n\leq x^{2r}} \frac{a_{2r,x}(n)}{n^{1/2}} \sum_{m\mid n} b(m,n) \lambda_j(m)\\
     =  \sum_{n \leq x^{2r}}\frac{a_{2r,x}(n)}{n^{1/2}}\sum_{m\mid n} b(m,n) \sum_{j} \frac{h(t_j;X,Y)}{L(1,\sym^2\phi_j)}\lambda_j(m).
\end{multline*}
Then by Lemma \ref{lemma:sumusesummationformula}, we get two parts. We write $G(X, Y) = \frac{2}{\sqrt{\pi}}XY + \mathcal{O}(Y^2)$. The diagonal term is
\[
\sum_{n \leq x^{2r}}\frac{a_{2r,x}(n)}{n^{1/2}} b(1,n)G(X,Y) = G(X,Y) \sum_{\substack{ n = \prod\limits_{i = 1}^{s}p_i^{\alpha_i}\\ 2 | \alpha_i , p_i \leq x}}\frac{a_{2r,x}(n)}{n^{1/2}}\prod_{i=1}^{s} \frac{\alpha_i !}{(\alpha_i/2)! (\alpha_i/2+1)!}.
\]
Let $\alpha_i = 2\beta_i$ and $n = m^2 = (\prod\limits_{i =1 }^{s}p_{i}^{\beta_i})^2$, then the sum is reduced to
\begin{equation*}
    \begin{aligned}
      \sum_{\substack{ m = \prod\limits_{i = 1}^{s}p_i^{\beta_i}\\   p_i \leq x}}\frac{a_{2r,x}(m^2)}{m}&\prod_{i=1}^{s} \frac{(2\beta_i) !}{(\beta)! (\beta+1)!}   = \sum_{\substack{ m = \prod\limits_{i = 1}^{s}p_i^{\beta_i}\\   p_i \leq x}}\frac{\binom{2r}{2\beta_1,\ldots,2\beta_s} \prod_{i=1}^{s} a(p_i)^{2\beta_i}}{m}\prod_{i=1}^{s} \frac{(2\beta_i) !}{(\beta)! (\beta+1)!}\\
      & \leq \sum_{\substack{ m = \prod\limits_{i = 1}^{s}p_i^{\beta_i}\\   p_i \leq x}}\frac{ \prod_{i=1}^{s} a(p_i)^{2\beta_i}}{m}\frac{ (2r!) }{2^r r!} \binom{r}{\beta_1,\ldots,\beta_s}\\
      & \leq \frac{ (2r)! }{2^r r!} (\sum_{p \leq x}\frac{a(p)^2}{p})^{r}.
    \end{aligned}
\end{equation*}
The off-diagonal is bounded by
\begin{equation*}
    \begin{aligned}
        \sum_{n \leq x^{2r}}&\frac{a_{2r,x}(n)}{n^{1/2}}\sum_{m\mid n} b(m,n) X^{\varepsilon}Y\\
        & \leq  X^{\varepsilon} Y\sum_{n \leq x^{2r}}\frac{a_{2r,x}(n)}{n^{1/2}} \tau(n) n\\
        & \leq X^{\varepsilon} Y\sum_{\substack{n = \prod\limits_{i}^{s}p_i^{\alpha_i} \\ \sum_{i}^{s}\alpha_i = 2r , p_i\leq x}}\frac{ \binom{2r}{\alpha_1,\ldots,\alpha_s} \prod_{i=1}^{s} |a(p_i)|^{\alpha_i}}{n^{1/2}} \tau(n) n\\
       & \leq X^{\varepsilon}Y x^{r + \varepsilon}(\sum_{p\leq x}|a(p)|)^{2r} \ll X^{\varepsilon}Y x^{5r} \ll X^{1/2+\varepsilon}Y.
    \end{aligned}
\end{equation*}
Then we complete the proof.

\end{proof}

\subsection{The estimate of density function}
Recall
\[
\lambda_{\Sym^2 \phi}(p)^2 = \lambda_{\phi}(p^4) - \lambda_{\phi}(p^2) + 1.
\]
We define 
\begin{multline}\label{Variance-L-function}
        \mathcal{V}(z_1,z_2 , x): = 2\mathcal{N}(2z_1,x) + 2\mathcal{N}(2z_2,x) \\+ 2\mathcal{N}(2\Re(z_1),x) + 2\mathcal{N}(2\Re(z_2),x) + 4\mathcal{N}(z_1 + z_2,x) +4\mathcal{N}(z_1 + \bar{z}_2,x).
\end{multline}

Then we have
\begin{multline}
   \label{l^2-norm-bound}
    \sum_{ p \leq x}\frac{|(\ell_1 +  \ell_2 \lambda_{\Sym^2\phi}(p) + \ell_3 (\frac{p^{-z_1} + p^{-\overline{z_1}  } + p^{-z_2} + p^{-\overline{z_2}  }}{2})) |^2}{p}  = (\ell_1^2 + \ell_2^2) \log \log x\\ + \frac{\ell^2_3}{4}\mathcal{V}(z_1,z_2,x)+ 2\ell_1\ell_3\mathcal{M}(z_1,z_2,x) +  \mathcal{O}(\log\log\log(X+t_{\phi})).
\end{multline}

We write
\[
\sigma(X)^2 = (\ell_1^2 + \ell_2^2) \log \log (X+t_{\phi}) + \frac{\ell^2_3}{4}\mathcal{V}(z_1,z_2,X+t_{\phi})+ 2\ell_1\ell_3\mathcal{M}(z_1,z_2,X+t_{\phi}). 
\]
Now we have the following key inequality in our proof.
\begin{proposition}\label{proposition-AXY}
 Assume GRH and GRC.
  Let $C\geq 1$ be fixed and $\varepsilon >0$ be sufficiently small.
  With the above notation, we have for
  $\sqrt{\log\log X}\leq V\leq C\frac{\log (X+t_\phi)}{\log\log (X+t_\phi)}$ that
  \begin{equation}
    \mathcal{A}_{X,Y}(V;
    X^{\frac{1}{\varepsilon V}})
    \ll
    G(X,Y)\Big(e^{-\frac{(1-2\varepsilon)V^2}
    {2\sigma(X)^2}} \log\log ( X+t_{\phi})
    +e^{-\frac{\varepsilon}{11}V\log V}\Big).
  \end{equation}
\end{proposition}

\begin{proof}
  We assume throughout that $\sqrt{\log\log X}\leq V\leq C\frac{\log (X+t_\phi)}{\log\log (X+t_\phi)}$.
  Set $x= X^{\frac{1}{\varepsilon V}}$
  and let
  $z=x^{\frac{1}{\log\log (X+t_\phi)}}$.
  Write $\mathcal{P}(t_j;x,x)
  =\mathcal{P}_1(t_j)+\mathcal{P}_2(t_j)$ where $\mathcal{P}_1(t_j)=\mathcal{P}(t_j;x,z)$.
  Also, let $V_1=(1-\varepsilon)V$ and $V_2=\varepsilon V$.
  If $\mathcal{P}(t_j;x,x)>V$ then
  \begin{equation}\label{eqn:case1}
   \textrm{ i) } \mathcal{P}_1(t_j)>V_1,
  \end{equation}
  or
    \begin{equation}\label{eqn:case2}
    \textrm{ ii) } \mathcal{P}_2(t_j)>V_2.
  \end{equation}
 We first consider case $\textrm{i)}$. Using Lemma \ref{lemma:sumusesummationformula} and note that the condition $ z^{\frac{1}{\varepsilon V}} \leq X^{\frac{1}{10 r}}$ that restricts $r \leq \frac{\varepsilon V}{10} \log\log( X + t_{\phi})$.
  We have the number of $X< t_j\leq X+Y$ for which \eqref{eqn:case1} holds is bounded by
  \begin{multline}
    \frac{1}{V_1^{2r}}
    \sum_{X< t_j\leq X+Y}
    \mathcal{P}_1(t_j)^{2r}
    \ll
     \frac{(\log\log X)^3}{V_1^{2r}}
    \sum_{t_j}
     \frac{h(t_j;X,Y)}
     {L(1,\sym^2\phi_j)}
    \mathcal{P}_1(t_j)^{2r}
    \\
    \ll  G(X,Y)(\log\log X)^3\frac{(2r)!}{V_1^{2r}r!2^r}
    ( \sum_{ p \leq z}\frac{|(\ell_1 +  \ell_2 \lambda_{\Sym^2\phi}(p) + \ell_3 (\frac{p^{-z_1} + p^{-\overline{z_1}  } + p^{-z_2} + p^{-\overline{z_2}  }}{2})) |^2}{p})^{r}
    \\
   + \frac{X^{1/2+\varepsilon}Y}{V_1^{2r}}
    \\ \ll
    G(X,Y)(\log\log X)^3\frac{(2r)!}{V_1^{2r}r!2^r}
    (\sigma(X)(1+o(1)))^{2r}
    +\frac{X^{1/2+\varepsilon}Y}{V_1^{2r}} 
    \\ \ll G(X,Y)(\log\log X)^3\Big(
    \frac{2r\sigma(X)^2(1+o(1))}{V_1^2 e}\Big)^r +\frac{X^{1/2+\varepsilon}Y}{V_1^{2r}} 
    \\ \ll G(X,Y)(\log\log X)^3\Big(
    \frac{2r\sigma(X)^2(1+o(1))}{V_1^2 e}\Big)^r
  \end{multline}
  where in the first step we applied $L(1,\sym^2\phi_j)\ll (\log\log X)^3$ under GRH and GRC, and in the third step we applied Stirling's formula.
  Take
  \[
    r = \left\{ \begin{array}{ll}
          \lfloor \frac{V_1^2}{2\sigma(X)^2} \rfloor, & \textrm{if } \sqrt{\log\log X} \leq V \leq \frac{\epsilon}{10} \sigma(X)^2 \log\log (X+t_{\phi}), \\
          \lfloor \frac{\epsilon V}{10} \rfloor, & \textrm{if } \frac{\epsilon}{10} \sigma(X)^2 \log\log (X+t_{\phi}) < V \leq C\frac{\log (X + t_{\phi})}{\log\log(X+t_\phi)}.
        \end{array}\right.
  \]
  Hence by Chebyshev inequality we have
\begin{multline}
      \#  \{ X<t_j\leq X+Y: \mathcal{P}_1(t_j)>V_1 \}\leq \frac{1}{V_1^{2r}}
    \sum_{X< t_j\leq X+Y}
    \mathcal{P}_1(t_j)^{2r}
      \\ 
      \ll
    G(X,Y)\Big(
    e^{-(1-2\varepsilon)\frac{V^2}
    {2\sigma(X)^2}} \log\log (X+t_{\phi})
    +e^{-\frac{\varepsilon}{11} V\log V}
    \Big).
\end{multline}
Now we consider case \textrm{ii)}. It remains to bound the number of $X<t_j\leq X+Y$ for which \eqref{eqn:case2} holds.
Take $r=\lfloor \frac{\varepsilon V}{10} \rfloor$.
As before, we use Lemma \ref{lemma:sumusesummationformula} and (\ref{l^2-norm-bound}) to estimate this quantity by
\begin{multline}
  \frac{1}{V_2^{2r}}
    \sum_{X< t_j\leq X+Y}
    \mathcal{P}_2(t_j)^{2r}
    \ll
     \frac{(\log\log X)^3}{V_2^{2r}}
    \sum_{t_j}
     \frac{h(t_j;X,Y)}
     {L(1,\sym^2u_j)}
    \mathcal{P}_2(t_j)^{2r}
    \\ \ll
    G(X,Y)(\log\log X)^{3}\frac{(2r)!}{r!}
    \Big(\frac{C}{V_2^2}\log\log\log (X+t_\phi)
    \Big)^r 
    \\ \ll
    G(X,Y)e^{-\frac{\varepsilon}{11}V\log V}.
\end{multline}
Then we complete the proof.
\end{proof}
\subsection{Proof of Theorem \ref{Mixed-moments-Soundararajan's-method-shifted}}
Recall the notations in Subsection \ref{Density-function-subsection}. Note that $X^{\varepsilon} \leq Y \leq X$ and use Proposition \ref{proposition-AXY}. Then we get
\begin{equation}
    \begin{aligned}
            &\sum_{X\leq t_j \leq X+Y}\mathcal{L}(t_j) = \int_{\mathbb{R}}e^{V}\mathcal{B}_{X,Y}(V) \dd V + \delta_{X \leq t_j \leq X+Y}D_{\phi}\mathcal{L}(t_{\phi})\\
            & \ll e^{\mu(X)}\int_{\mathbb{R}}e^{V}\mathcal{B}_{X,Y}(V + \mu(X)) \dd V + Ye^{\frac{c\log(X+t_{\phi})}{\log\log(X+t_{\phi})}}\\
            & \ll e^{\mu(X)}\int_{\sqrt{\log\log(X+t_\phi)}}^{\frac{c\log(X+t_{\phi})}{\log\log(X+t_{\phi})}}e^{V} G(X,Y)\Big(e^{-\frac{(1-2\varepsilon)V^2}
    {2\sigma(X)^2}} \log\log (X+t_{\phi})
    +e^{-\frac{\varepsilon}{11}V\log V}\Big) \dd V \\
    & \quad\quad\quad + Ye^{\frac{c\log(X+t_{\phi})}{\log\log(X+t_{\phi})}}.\\
    \end{aligned}    
\end{equation}
The first part above is bounded by
\begin{equation}
    \begin{aligned}
   &G(X,Y)e^{\mu(X)+ \frac{\sigma(X)^2}{2(1-2\varepsilon)}}(\log (X+t_\phi))^{\varepsilon}\\
    & \ll XY (\log (X+t_\phi))^{\varepsilon} \exp(-\frac{1}{2}(\ell_1+\ell_2)\log\log(X+t_{\phi}) - \frac{\ell_3}{2}\mathcal{M}(z_1,z_2 , X+t_\phi))\\
    &\quad\times \exp\left( \frac{(\ell_1^2 + \ell_2^2)}{2} \log \log (X+t_{\phi}) + \frac{\ell^2_3}{8}\mathcal{V}(z_1,z_2,X+t_{\phi})+ \ell_1\ell_3\mathcal{M}(z_1,z_2,X+t_{\phi})\right)\\
    & \ll  XY (\log X)^{\frac{\ell_1(\ell_1 -1)}{2}+\frac{\ell_2(\ell_2 -1)}{2} + \varepsilon}\exp(\ell_3(\ell_1 - \frac{1}{2})\mathcal{M}(z_1,z_2 , X) + \frac{\ell_3^2}{8}\mathcal{V}(z_1,z_2 , X)).
    \end{aligned}
\end{equation}
The second part is bounded by
 \[
 \mathcal{O}(YX^{\varepsilon}).
 \]
 
  Then we complete the proof of Theorem \ref{Mixed-moments-Soundararajan's-method-shifted}.
  
\appendix

\section{\label{sec:6}Remarks on the Conjecture \ref{Conjecture-Joint-value-distribution}}This section presents some remarks on our conjecture \ref{Conjecture-Joint-value-distribution}. 

At first, we discuss the distinctions between the pure fourth moment problem (see \cite{MR3647437} and \cite{MR4184616}) and the mixed fourth moments problems of automorphic forms (for example,  \cite{hua2024jointvaluedistributionheckemaass} and our work). Although both types of problems employ similar strategies, using Plancherel formula and triple product formula to get the certain moments of  $L$-functions. For $f , g$ be two distinct Hecke--Maass forms. A natural question arises that why the asymptotic formulas of
\[
\langle f^2 ,f^2\rangle, \quad\quad \langle f^2,g^2\rangle
\]
are different. More precisely, we can ask why the two moments of $L$-functions
\[
\sum_{t_j \leq X}L(\frac{1}{2},\phi_j)L(\frac{1}{2},\Sym^2 f\times\phi_j)
\]
and 
\[
\sum_{t_j \leq X}L(\frac{1}{2},\phi_j)L(\frac{1}{2},\Sym^2 f\times\phi_j)^{1/2}L(\frac{1}{2},\Sym^2 g\times\phi_j)^{1/2}
\]
are different.

The first sum, weighted by suitable test functions, admits an asymptotic formula obtained by spectral method and GLH.  For the second one Hua--Huang--Li got a $\log$-saving bound by using Soundararajan's method under GRH and GRC. Clearly, such a $\log$-saving bound cannot hold for the first sum, since a sharp asymptotic formula already exists. 

The key point is how to distinguish two $f \neq g$ in $L$-functions when applying Soundararajan's method. Essentially, they use the Selberg orthogonal relation. Roughly speaking, that is
\begin{equation}
    \begin{aligned}
\sum_{p \leq X}\frac{a_f(p)b_g(p)}{p} =  \mathcal{O}(\log\log\log (t_{f} + t_{g}))
    \end{aligned}
\end{equation}
where $ a_{f}(p
)\in \{\lambda_f(p) , \lambda_{\Sym^2 f}(p),\lambda_{\Sym^4 f}(p)\}$ and $ b_{g}(p
)\in \{\lambda_g(p) , \lambda_{\Sym^2 g}(p)\}$.

As a natural generalization, we need distinguish $f$ and $E_{T}$ in the proof of Theorem \ref{Nonequidistribution-theorem} and  Theorem \ref{equidistribution-phenomenon}. We need an estimate of the form 
\begin{equation}
    \begin{aligned}
\sum_{p \leq X}\frac{a_{f}(p)p^{-iT}}{p} = \mathcal{O}(\log\log\log (T+ t_{f})). 
    \end{aligned}
\end{equation}

A more subtle issue arises in distinguishing two Eisenstein series even if the spectral parameters are almost equivalent. Notably, the relation
\begin{equation}
    \begin{aligned}
\sum_{p \leq X}\frac{p^{iT_1}p^{-iT_2}}{p} = \mathcal{O}(\log\log\log  (T_1+T_{2})) 
    \end{aligned}
\end{equation}
does not hold for some small $|T_1-T_2|$ depending on $X$! In fact, we only have
\begin{equation}
    \begin{aligned}
    \sum_{p \leq X}\frac{1}{p^{1+iT}} =  \mathcal{O}(\log\log\log T) +  \left\{\begin{array}{lr}
             \log\log X, \quad & |T| \leq \frac{1}{\log X},  \\
           -\log |T|, \quad  & \frac{1}{\log X}\leq |T| \leq 1,\\
            0, \quad & |T|\geq 1.
        \end{array}
        \right.
    \end{aligned}
\end{equation}

It seems that we can't use the  tools directly from  bounds of mixed moments of $L$-functions to detect the fourth moment (suitable regularized type)
\[
\langle |E_{t} |^2 , |E_{t} |^2\rangle,\quad\quad \langle |E_{t} |^2 , |E_{\tau} |^2\rangle.
\]
In such cases, one may only obtain upper bounds of the same order as those for pure fourth moments.

Therefore, whether the independent value distribution fails in this subtle situation (for example, when $|t-\tau|$ is very small)? To support our conjecture, we examine the simplest case, known as the decorrelation of Eisenstein series, and establish an asymptotic formula.  Indeed, our analysis shows that two Eisenstein series with very close spectral parameters cannot be distinguished from the perspective of value distribution.
 \begin{proposition}\label{decorrelation-Eisenstein-series}Let $\psi(z) \in C_{c}^{\infty}(\mathbb{X})$ fixed. Assume $1 < \tau < t$. Then we get
\begin{equation}
    \begin{aligned}    \int_{\mathbb{X}}&\psi(z)\widetilde{E_{t}(z)}\widetilde{E_{-\tau}(z)} \dd \mu z =  \left( \sqrt{\frac{\vol(\mathbb{X})}{\log(\frac{1}{4}+t^2)}} \sqrt{\frac{\vol(\mathbb{X})}{\log(\frac{1}{4}+\tau^2)}}\right)\times \\&    \left\{\begin{array}{lr} \mathcal{O}_{\psi}((\log t)^{-100}), &   |t-\tau| \geq (\log t)^{\varepsilon},  \\
          \mathcal{O}_{\psi}((\log t )^{2/3+\varepsilon} ),   & (\log t)^{-2/3-\varepsilon} \leq   |t - \tau| \leq (\log t)^{\varepsilon},\\
         \frac{6}{\pi}\frac{\sin \left((t-\tau)\log \tau\right)}{t -\tau}\int_{\mathbb{X}}\psi(z)\dd \mu z  + \mathcal{O}_{\psi}((\log t )^{2/3+\varepsilon} ),    &  |t - \tau| \leq (\log t)^{-2/3-\varepsilon}.
        \end{array}
        \right.
    \end{aligned}
\end{equation}    as $\tau$ goes to infinite.
\end{proposition}
\begin{proof}By applying regularized Plancherel formula Lemma \ref{Prop_Regular_Planch}, we decompose $\psi(z)$ to the discrete spectrum, continuous spectrum and regularized term. The first two parts are bounded by certain subconvexity bound of $L$-functions with power saving error terms. Therefore, we only need estimate the regularized part
\begin{equation}\label{Regularized-part}
    \left( \sqrt{\frac{\vol(\mathbb{X})}{\log(\frac{1}{4}+t^2)}} \sqrt{\frac{\vol(\mathbb{X})}{\log(\frac{1}{4}+\tau^2)}}\right)\int_{\mathbb{X}}\psi(z)\Phi_{t,\tau}(z) \dd \mu z 
\end{equation} 
where
    \begin{multline}
  \Phi_{t,\tau}(z) = 2\Re\left(\frac{\xi(1+2it)\xi(1-2i\tau)}{|\xi(1+2it)\xi(1-2i\tau)|}E(z , 1+i(t - \tau))\right)\\
  + 2\Re\left(\frac{\xi(1+2it)\xi(1+2i\tau)}{|\xi(1+2it)\xi(1+2i\tau)|}E(z , 1 + i(t+\tau))\right).
    \end{multline}
For any $T \geq 0 $ and $z \in \mathcal{F}$, from Young's work \cite[Equation 3.9]{Young_2018} we have
\begin{equation}\label{Young's-work}
    E(z , 1+iT) - \left(y^{1+2iT} + \frac{\xi(1+2iT)}{\xi(2+2iT)}y^{iT}\right) \ll (1+T)^{1+\varepsilon}y^{-1}.
\end{equation}
    
Note that for any $j > 0$ and $ T  \geq 1$, we have
\[
(1+T^2)^{j}\langle \psi, E(\cdot , 1 + iT)\rangle = \langle \psi , \Delta^{j} E(\cdot , 1 + iT)  \rangle = \langle \Delta^{j}\psi , E(\cdot , 1 + iT)  \rangle \ll_{\psi}T^{1+\varepsilon}.  
\]
Then we get
\begin{equation}\label{partial-integration}
\langle \psi, E_{T}\rangle \ll_{\psi} (1 + T)^{-2j +1+\varepsilon}
\end{equation}
for any $j >0$. By using (\ref{partial-integration}), we get the contribution of $|t - \tau| \geq (\log t)^{\varepsilon}$ is $\mathcal{O}_{\psi}((\log t)^{-100})$ by taking $j$ sufficiently large. Hence we need deal with $|t - \tau| \ll (\log t)^{\varepsilon}$. We use the expansion of $E(z,s)$ as $s \rightarrow 1$ that
\begin{equation}\label{near-s=1}
    E(z , 1+iT) = \frac{3}{\pi iT} + \mathcal{O}(1).
\end{equation}
as $T \rightarrow 0$.

When $ (\log t)^{-2/3-\varepsilon}\leq |t-\tau| \leq (\log t)^{\varepsilon}$, we get
\[
2\Re\left(\frac{\xi(1+2it)\xi(1-2i\tau)}{|\xi(1+2it)\xi(1-2i\tau)|}E(z , 1+i(t - \tau))\right) \ll_{\psi} (\log t)^{2/3 + \varepsilon}
\]
trivially by (\ref{Young's-work}), (\ref{near-s=1}). The constant depends on the compact support of $\psi(z)$.

For $ |t - \tau| \ll (\log t)^{-2/3-\varepsilon}$. Then we need study the rotations in 
\begin{equation}\label{near-parameter} 2\Re\left(\frac{\xi(1+2it)\xi(1-2i\tau)}{|\xi(1+2it)\xi(1-2i\tau)|}\frac{3}{\pi i (t - \tau)}\right).
\end{equation}
The rotation is
\[
e^{i\arg(\xi(1+2it)) + i \arg(\xi(1-2i\tau))}.
\]
Let $f(x) = \arg\xi(1+2ix)$. Since $\xi(1+2ix)$ has no zero for any $x \in \mathbb{R}$, we get $\arg\xi(1+2ix)$ is harmonic as the imaginary part of $\log \xi(1+2ix)$. Then we get 
\[
f(t) = f(\tau) + f'(\tau)(t - \tau) + \frac{f''(\theta)}{2}(t - \tau)^2
\]
where $ \tau\leq \theta \leq t$. Now the rotation is reduced to
\[
\exp\left(i(t - \tau)\frac{\dd }{\dd x}\arg\xi(1+2ix){\Big|}_{x = \tau} + i\frac{(t-\tau)^2}{2}\frac{\dd^2 }{\dd x^2}\arg\xi(1+2ix){\Big|}_{x = \theta} \right).
\]
We have
\begin{equation}
    \begin{aligned}
\frac{\dd }{\dd x}\arg\xi(1+2ix) &= \frac{\dd }{\dd x}\arg\pi^{-\frac{1+2ix}{2}} + \frac{\dd }{\dd x}\arg\Gamma(\frac{1}{2}+ix)+\frac{\dd }{\dd x}\arg\zeta(1+2ix)\\
& = \Im\left(\frac{\log \pi^{-i} \cdot\pi^{-ix}}{\pi^{-ix}}\right) + \Im\left(i\frac{\Gamma'}{\Gamma}(\frac{1}{2}+ix)\right) + \Im\left(\frac{\zeta'}{\zeta}(1+2ix)\right)\\
& = \log (\frac{1}{2}+ix) + \mathcal{O}((\log x)^{2/3}(\log\log x)^{1/3})
    \end{aligned}
\end{equation}
from Stirling's formula $\frac{\Gamma'}{\Gamma}(\frac{1}{2}+ix) = \log (\frac{1}{2}+ix) + \mathcal{O}(\frac{1}{|x|})$ and  Vinogradov–Korobov bound $\frac{\zeta'}{\zeta}(1+2ix) \ll (\log x)^{2/3}(\log\log x)^{1/3}$. We also get
\[
\frac{\dd^2}{\dd x^2}\arg\xi(1+ix) \ll (\log t)^{4/3+\varepsilon}
\]
from $(\frac{\Gamma'}{\Gamma})'((\frac{1}{2}+ix))  = \mathcal{O}(1)$ and $(\frac{\zeta'}{\zeta})'(1+2ix) \ll (\log x)^{4/3+\varepsilon}$.
Then we get
\begin{multline}
    \exp\left(i(t - \tau)\frac{\dd }{\dd x}\arg\xi(1+2ix){\Big|}_{x = \tau} + i(t - \tau)^2\frac{\dd^2 }{\dd x^2}\arg\xi(1+2ix){\Big|}_{x = \theta}\right) \\= \exp(i(t-\tau)\log\tau
) \Big(1 + \mathcal{O}(|t-\tau|(\log t)^{2/3+\varepsilon})\Big).
\end{multline}
The contribution of the error term in (\ref{near-parameter}) is
\[
\mathcal{O}_{\psi}((\log t )^{2/3+\varepsilon}).
\]
The contribution of the main term is  
\[
2\Re\left(\exp(i(t-\tau)\log\tau
)\frac{3}{\pi i (t - \tau)}\right)  = \frac{6}{\pi}\frac{\sin \left((t-\tau)\log \tau\right)}{t -\tau}.
\]

Hence, we get
\[
\int_{\mathbb{X}}\psi(z)\Phi_{t,\tau}(z) \dd \mu z = \frac{6}{\pi}\frac{\sin \left((t-\tau)\log \tau\right)}{t -\tau}\int_{\mathbb{X}}\psi(z)\dd\mu z  + \mathcal{O}_{\psi}((\log t )^{2/3+\varepsilon} )
\]
in this case.

Then we complete the proof from (\ref{Regularized-part}).
\end{proof}

\section{\label{sec:7}Alternative computation of main term (\ref{Main-term}) in Theorem \ref{Nonequidistribution-theorem}}We give another view of the complex main term in Theorem \ref{Nonequidistribution-theorem} instead of regularized inner product. We can see that this term completely comes from the cusp.

Note that the difference of $L^4$-norm is
\begin{equation}\label{difference}
\begin{aligned}
        \langle \phi^2 , |E_{T}|^2 - |E_{T}^{A}|^2\rangle &= \overline{c_{T}^2}\int_{\mathcal{C}_{A}}\phi^2(z) \overline{e(y,\frac{1}{2}+iT)(2 E_{T}(z) - e(y, \frac{1}{2}+iT))\frac{\dd x\dd y}{y^2}}\\
        &= \overline{c_{T}^2}\int_{\mathcal{C}_{A}}\phi^2(z) \overline{e(y,\frac{1}{2}+iT)(e(y, \frac{1}{2}+iT) + 2E_{T}^A(z))}\frac{\dd x\dd y}{y^2}
\end{aligned}
\end{equation}
where $c_T = \frac{\xi(1+2iT)}{|\xi(1+2iT)|}$. We have three key terms. The diagonal term is
\begin{equation}
\mathcal{I}' := 2\int_{\mathcal{C}_A}\phi^2(z)y \frac{\dd x \dd y}{y^2}.
\end{equation}
The other two terms from $e(y,\frac{1}{2}+iT)$ have type
\[
\mathcal{J}' : = \frac{\xi(1-2iT)}{\xi(1+2iT)}\int_{\mathcal{C}_A}\phi^2(z)y^{1 - 2iT} \frac{\dd x \dd y}{y^2} + \frac{\xi(1+2iT)}{\xi(1-2iT)}\int_{\mathcal{C}_A}\phi^2(z)y^{1 + 2iT} \frac{\dd x \dd y}{y^2}.
\]
The truncated part is
\begin{equation}
    \begin{aligned}
        \mathcal{K}' &:= 2\overline{c_{T}^2}\int_{\mathcal{C}_{A}}\phi^2(z) \overline{e(y,\frac{1}{2}+iT)E_{T}^A(z)} \frac{\dd x\dd y}{y^2}\\
        & = 2\int_{\mathcal{C}_A}\phi^2(z)(y^{1/2-iT}+\frac{\xi(1+2iT)}{\xi(1-2iT)}y^{1/2+iT})E_{T}^{A}(z) \frac{\dd x\dd y}{y^2}\\
        & = 2\sum_{j}\langle \phi^2 , \phi_j\rangle\int_{\mathcal{C}_A}\phi_j(z)(y^{1/2-iT}+\frac{\xi(1+2iT)}{\xi(1-2iT)}y^{1/2+iT})E_{T}^{A}(z) \frac{\dd x\dd y}{y^2} + continuous.
    \end{aligned}
\end{equation}
\begin{remark}In fact, we use spectral decomposition to $\phi^2$ then $\mathcal{K}'$ is almost $\mathcal{J}(A)$. It does not contribute the main term under GRH and GRC.
\end{remark}
We find that $\mathcal{I}'$ and $\mathcal{J}'$ consist of the part which containing $t_{\phi}$ in main term $\mathcal{R}(\phi, E_{T})$  .

\begin{proposition}\label{difference-fourth-moment}Assume GLH. We have
    \begin{multline}
 \langle \phi^2 , |E_{T}|^2 - |E_{T}^{A}|^2\rangle = \mathcal{K}'+ \mathcal{O}_A(1) + \frac{\Lambda'(1,\Sym^2 \phi)}{\xi(2)\Lambda(1,\Sym^2\phi)}\\+ \frac{\xi(2iT)}{\xi(1+2iT)}\langle \phi^2 , E(z,1+2iT)\rangle + \overline{\frac{\xi(2iT)}{\xi(1+2iT)}}\langle \phi^2 , E(z,1-2iT)\rangle
\end{multline}
\end{proposition}
The proposition is from the two lemmas as follow.
\begin{lemma}\label{lemma-contribution-of-I}Assume GLH. We have
\begin{equation}\label{contribution-of-I}
    \mathcal{I}' = \frac{\Lambda'(1,\Sym^2 \phi)}{\xi(2)\Lambda(1,\Sym^2\phi)} + \mathcal{O}(\log A + A^{1/2}t_{\phi}^{-1/2+\varepsilon}).
\end{equation}
\end{lemma}
\begin{lemma}\label{lemma-contribution-of-J}Assume GLH. We have
    \begin{equation}\label{contribution-of-J}
\mathcal{J}' =   \frac{\xi(2iT)}{\xi(1+2iT)}\langle \phi^2 , E(z,1+2iT)\rangle + \overline{\frac{\xi(2iT)}{\xi(1+2iT)}}\langle \phi^2 , E(z,1-2iT)\rangle+ \mathcal{O}(A^{1/2}t_{\phi}^{-1/2+\varepsilon}).  
\end{equation}
\end{lemma}
From Proposition \ref{difference-fourth-moment} and (\ref{I}), we deduce the all main term in Theorem \ref{Nonequidistribution-theorem}.

\begin{proof}[The proof of Lemma \ref{lemma-contribution-of-I}]
Note that 
\[
\int_{\mathcal{C}_A}\phi^2(z)y \frac{\dd x \dd y}{y^2} = \int_{A}^{\infty}\int_{-1/2}^{1/2}\phi^2(z)y \frac{\dd x \dd y}{y^2}.
\]
Opening the Maass form by Fourier expansion and integrate over $x$, we get
\[
8|\rho_{\phi}(1)|^2 \sum_{n\geq 1}\lambda_{\phi}(n)^2 \int_{A}^{\infty}K_{it_{\phi}}(2\pi ny)^2 y \frac{\dd y}{y}.
\]

By Mellin inversion, we get
\[
K_{i\phi}(2\pi ny)^2 = \frac{1}{2\pi i }\int_{(\sigma)}2^{s-3}\frac{\Gamma(\frac{s}{2})^2\Gamma(\frac{s+2it_{\phi}}{2})\Gamma(\frac{s-2it_{\phi}}{2})}{\Gamma(s)} (2\pi n y)^{-s} \dd s
\]
where $\sigma = 3$. Then the main term is 
\[
|\rho_{\phi}(1)|^2 \frac{1}{2\pi i }\int_{(\sigma)}\pi^{-s}\frac{\Gamma(\frac{s}{2})^2\Gamma(\frac{s+2it_{\phi}}{2})\Gamma(\frac{s-2it_{\phi}}{2})}{\Gamma(s)}\frac{L(s,\Sym^2\phi)\zeta(s)}{\zeta(2s)}\frac{A^{1-s}}{s-1} \dd s.
\]
By Stirling formula we can truncate the integral to $|\Im s| \leq 2t_{\phi}+t_{\phi}^{\varepsilon}$. Shifting the integration to the line $\Re(s) = 1/2$, we need pick the second order pole at $s = 1$. Let 
\[
f(s) = \pi^{-s/2}\frac{\Gamma(\frac{s}{2})\Gamma(\frac{s+2it_{\phi}}{2})\Gamma(\frac{s-2it_{\phi}}{2})}{\Gamma(s)}\frac{L(s,\Sym^2\phi)A^{1-s}}{\zeta(2s)} = \frac{A^{1-s}\Lambda(s, \Sym^2 \phi)}{\xi(2s)}.
\]
Note the classical expansion 
\[
\frac{1}{s-1}\xi(s) = \frac{1}{(s-1)^2} + \frac{c_0}{s-1} + \mathcal{O}(1)
\]
when $s \rightarrow 1^+$. We need the Taylor expansion of $f(s)$ at $s=1$ that is
\[
f(s) = f(1) + f'(1)(s - 1) + \mathcal{O}(1)
\]
where $f(1) = \frac{\Lambda(1,\Sym^2\phi)}{\xi(2)}$ and
\begin{equation}
    \begin{aligned}       
f'(1)  = & \lim\limits_{s \rightarrow 1} \frac{[(\log A)A^{1-s}\Lambda(s,\Sym^2 \phi) +A^{1-s}\Lambda'(s,\Sym^2 \phi)]\xi(2s) - A^{1-s}\Lambda(s,\Sym^2\phi)2\xi'(2s) } {(\xi(2s))^2}.
    \end{aligned}
\end{equation}
We get
\[
f'(1) = \frac{\Lambda'(1,\Sym^2 \phi)}{\xi(2)} + C \log A\Lambda(1,\Sym^2\phi)
\]
where $C$ is a absolute constant. Recall that $|\rho_{\phi}(1)|^2 = \frac{1}{2\Lambda(1,\Sym^2\phi)}$ and we get the contribution of residue equals to
\begin{equation}\label{residue-term}
    \frac{1}{2\Lambda(1,\Sym^2 \phi)}[f'(1) + f(1)a_0] = \frac{\Lambda'(1,\Sym^2 \phi)}{2\xi(2)\Lambda(1,\Sym^2\phi)} + \log A\times \mathcal{O}(1).
\end{equation}
And the integration at line $\Re(s) = 1/2$ is bounded by
\begin{equation}\label{integral-term}
    A^{1/2}\frac{1}{t_{\phi}^{1/4}}\int_{0 < t \leq 2t_{\phi}+t_{\phi}^{\varepsilon}}\frac{|L(\frac{1}{2}+it,\Sym^2\phi)||\zeta(\frac{1}{2}+it)|}{(1 + |t-2t_{\phi}|)^{1/4}(1 + |t|)}\dd t.
\end{equation}
By GLH, we get the contribution is $\mathcal{O}(A^{1/2}t_{\phi}^{-1/2+\varepsilon})$.
In conclusion, by  (\ref{residue-term}) and (\ref{integral-term}) we get
\[
\int_{\mathcal{C}_A}\phi^2(z)y \frac{\dd x \dd y}{y^2} = \frac{\Lambda'(1,\Sym^2 \phi)}{2\xi(2)\Lambda(1,\Sym^2\phi)} + \mathcal{O}(\log A) + \mathcal{O}(A^{1/2}t_{\phi}^{-1/2+\varepsilon}).
\]
Then 
\begin{equation}
    \mathcal{I}' = \frac{\Lambda'(1,\Sym^2 \phi)}{\xi(2)\Lambda(1,\Sym^2\phi)} + \mathcal{O}(\log A + A^{1/2}t_{\phi}^{-1/2+\varepsilon}).
\end{equation}
    
\end{proof}
\begin{proof}[The proof of Lemma \ref{lemma-contribution-of-J}]
    We consider the first one and another one is the conjugation.  Similarly, we use Fourier expansion directly and deduce
\[
|\rho_{\phi}(1)|^2 \frac{1}{2\pi i }\int_{(\sigma)}\pi^{-s}\frac{\Gamma(\frac{s}{2})^2\Gamma(\frac{s+2it_{\phi}}{2})\Gamma(\frac{s-2it_{\phi}}{2})}{\Gamma(s)}\frac{L(s,\Sym^2\phi)\zeta(s)}{\zeta(2s)}\frac{A^{1-s - 2iT}}{s-1 + 2iT} \dd s.
\]
In this case $s= 1$ and $s = 1 + iT$ are only two simple poles, then we shift the integral to the line $\Re(s) = 1/2$. And the integral is bounded by $\mathcal{O}(t_{\phi}^{-1/2+\varepsilon})$ as the first one. The residue at $s= 1$ is
$\frac{A^{\mp iT}}{2\xi(2) T} \ll 1$. And the residue at $s = 1-2iT$ is
\[
\frac{\Lambda(1 -2iT,\Sym^2 \phi)\xi(1-2iT)}{2\Lambda(1,\Sym^2 \phi)\xi(2-4iT)} = \langle \phi^2 , E(\cdot , 1 + 2iT)\rangle.
\]
Similarly, we get the other one. Then we get
\begin{equation}
\mathcal{J}' =   \frac{\xi(2iT)}{\xi(1+2iT)}\langle \phi^2 , E(z,1+2iT)\rangle + \overline{\frac{\xi(2iT)}{\xi(1+2iT)}}\langle \phi^2 , E(z,1-2iT)\rangle+ \mathcal{O}(A^{1/2}t_{\phi}^{-1/2+\varepsilon}).  
\end{equation}
\end{proof}

\section*{Acknowledgements}

The author would like to thank  Prof. Bingrong Huang for his encouragement and helpful discussions.

\addcontentsline{toc}{section}{参考文献}
\phantomsection

\bibliographystyle{alpha}



\end{document}